\documentclass[11pt,a4paper]{amsart}
\usepackage[utf8]{inputenc}
\usepackage[T1]{fontenc}
\usepackage[english]{babel}
\usepackage{lmodern,microtype}
\usepackage[margin=25mm]{geometry}
\usepackage{mathtools,amssymb,mathrsfs}
\usepackage{enumitem}
\usepackage{xcolor}
\usepackage[colorlinks=true,linkcolor=blue!40!black,
 citecolor=blue!40!black,urlcolor=blue!50!black]{hyperref}
\hypersetup{
 pdftitle={Asymptotic multiplicity sequences of graded families},
 pdfauthor={T. H. Freitas and V. H. Jorge Perez},
 pdfsubject={Asymptotic multiplicity sequences of graded families},
 pdfkeywords={multiplicity sequence, graded family, filtration, Segre number, j-multiplicity}}
\makeatletter
\@ifundefined{subjclassname@2020}{%
  \@namedef{subjclassname@2020}{\textup{2020} Mathematics Subject Classification}}{}
\makeatother
\setlist[enumerate]{itemsep=3pt,topsep=5pt}
\numberwithin{equation}{section}
\theoremstyle{plain}
\newtheorem{theorem}{Theorem}[section]
\newtheorem{proposition}[theorem]{Proposition}
\newtheorem{lemma}[theorem]{Lemma}
\newtheorem{corollary}[theorem]{Corollary}

\theoremstyle{definition}
\newtheorem{definition}[theorem]{Definition}

\newtheorem{example}[theorem]{Example}

\newtheorem{problem}[theorem]{Problem}
\theoremstyle{remark}
\newtheorem{remark}[theorem]{Remark}
\newcommand{\m}{\mathfrak m}
\newcommand{\p}{\mathfrak p}
\newcommand{\I}{\mathcal I}
\newcommand{\J}{\mathcal J}
\newcommand{\K}{\mathcal K}

\newcommand{\cvec}{\mathbf c}

\newcommand{\htop}{\operatorname{ht}}

\newcommand{\Min}{\operatorname{Min}}
\newcommand{\gr}{\operatorname{gr}}

\newcommand{\lowc}{\underline c}
\newcommand{\upc}{\overline c}

\title[Asymptotic multiplicity sequences]{Asymptotic multiplicity sequences of graded families}
\author{T. H. Freitas}
\address{Universidade Tecnol\'ogica Federal do Paran\'a, 85053-525, Guarapuava-PR, Brazil}
\email{freitas.thf@gmail.com}

\author{V. H. Jorge P\'erez}
\address{Universidade de S{\~a}o Paulo -- ICMC, 13560-970, S{\~a}o Carlos-SP, Brazil}
\email{vhjperez@icmc.usp.br}
\date{}
\keywords{Multiplicity sequence, graded family of ideals, filtration, $j$-multiplicity, integral closure. The second author CNPq-Brazil, grant 304851/2025-6 and -FAPESP-Brazil, grant 2025/21618-0}

\begin{document}
\begin{abstract}
We study asymptotic multiplicity sequences of graded families of arbitrary ideals in Noetherian local rings. We characterize convergence of the zeroth component and prove that the first positive component not forced to vanish always converges under natural equidimensionality and catenarity assumptions. In contrast, higher components may exhibit arbitrary asymptotic behavior: even in $k[[x,y]]$, integrally closed filtrations with fixed radical, height, and analytic spread can realize any prescribed lower and upper limits. For Noetherian graded filtrations, however, all asymptotic components exist, are finite and rational, and satisfy Rees-type and valuative criteria for integral dependence. In the reduced pure-dimensional complex analytic setting, a common principalizing modification provides a second convergence mechanism without finite-generation assumptions. We also develop a diagonal mixed theory: for graded $\mathfrak m$-primary families, the top diagonal components are sums of Cutkosky's mixed multiplicities, while on a common analytic model all diagonal mixed components converge. The latter result extends the one-family common-model theorem.

\end{abstract}
\maketitle

\section{Introduction}
\label{sec:introduction}

Multiplicity provides one of the most effective numerical ways of
capturing the local complexity of ideals and singularities. In the
classical setting, Hilbert--Samuel multiplicity records the leading
growth of the colengths of powers of an $\mathfrak m$-primary ideal
and plays a central role in local algebra and intersection theory.
Beyond the primary case, however, the relevant colengths may be
infinite and a single numerical invariant is no longer sufficient.
This naturally leads to finer multiplicity invariants adapted to
arbitrary ideals.

Achilles and Manaresi \cite{AM97} introduced the multiplicity sequence
through the Hilbert polynomial of a naturally associated bigraded
algebra (see also \cite{AM06}). Their construction extends
Hilbert--Samuel multiplicity to arbitrary ideals and incorporates
other important invariants, including the $j$-multiplicity. Further
structural properties of the multiplicity sequence for arbitrary
ideals were developed by Callejas--Bedregal and Jorge P\'erez
\cite{CBJPAdd}. With the codimension indexing used in \cite{PTUV},
we write
\[
\cvec(I;R)=(c_0(I;R),\ldots,c_d(I;R)),
\qquad d=\dim R,
\]
so that $c_d(I;R)=j(I;R)$, while for an $\mathfrak m$-primary ideal
the sequence reduces to $(0,\ldots,0,e(I;R))$.

The algebraic construction has a geometric counterpart. Gaffney and
Gassler \cite{GaffneyGassler} developed Segre numbers in complex
analytic geometry, and Achilles and Rams \cite{AchillesRams}
identified them with the corresponding components of the
Achilles--Manaresi sequence. More recently, Polini, Trung, Ulrich, and
Validashti \cite{PTUV} showed that the multiplicity sequence also
detects integral dependence: for comparable ideals in suitable
equidimensional and universally catenary local rings, equality of the
multiplicity sequences characterizes integral dependence. Related
developments involving multidegrees, families, and integral dependence
appear in \cite{CRPU}. Thus the multiplicity sequence carries both
algebraic and intersection-theoretic information and provides a
natural refinement of classical multiplicity theory.

A parallel asymptotic theory has developed for graded families. If
$\mathcal I=\{I_n\}_{n\geq0}$ is a graded family of
$\mathfrak m$-primary ideals, Cutkosky proved that
\[
e(\mathcal I;R)
=
\lim_{n\to\infty}\frac{e(I_n;R)}{n^d}
\]
exists over an arbitrary Noetherian local ring
\cite[Theorem~1.3]{Cut}. He also constructed the corresponding mixed
multiplicities for several graded $\mathfrak m$-primary families
\cite[Theorem~1.6]{Cut}. For arbitrary ideals, the homogeneity
relation $c_i(I^n;R)=n^ic_i(I;R)$
suggests the normalization $\frac{c_i(I_n;R)}{n^i}.$
This leads to the basic question studied in this paper: for a graded
family $\mathcal I=\{I_n\}_{n\geq0}$ of arbitrary ideals, when do the
limits
\[
c_i(\mathcal I;R)
=
\lim_{n\to\infty}\frac{c_i(I_n;R)}{n^i},
\qquad 0\leq i\leq d,
\]
exist and remain finite? Unlike the primary case, different components
can display substantially different asymptotic behavior.

The first possible components exhibit a strong rigidity.
Proposition~\ref{prop:zeroth-component} characterizes the existence of
$c_0(\mathcal I;R)$ in terms of additive semigroups attached to the
top-dimensional minimal primes of $R$. In particular,
$c_0(I_n;R)$ is eventually periodic for every graded family and is
constant for a graded filtration. If $R$ is equidimensional and
catenary and $h=\operatorname{ht}I_1>0$, then
Theorem~\ref{thm:height} shows that
$c_i(\mathcal I;R)=0$ for $i<h$, while
$c_h(\mathcal I;R)$ always exists and satisfies $0\leq c_h(\mathcal I;R)\leq c_h(I_1;R).$
Thus the first component not forced to vanish is governed by the
Hilbert--Samuel multiplicities of localized graded primary families.
Under certain assumptions,
Corollary~\ref{cor:fixed-radical-equimultiple} shows that the entire
asymptotic sequence is concentrated in this single degree.

The higher components behave very differently. Already on the
two-dimensional regular local ring $k[[x,y]]$, we construct
integrally closed graded filtrations with fixed radical, height, and
analytic spread for which the normalized top component has arbitrary
prescribed lower and upper limits. More precisely,
Corollary~\ref{ex:realization} realizes every pair
$0\leq A\leq B\leq+\infty$. Hence neither the filtration property nor
integral closedness is sufficient to force convergence. Moreover,
Remark~\ref{ex:non-noetherian-algebra} shows that all real valuations
nonnegative on $R$ have the same asymptotic order on these examples,
independently of the prescribed values $A$ and $B$. Thus asymptotic
valuation orders alone do not determine the higher components of the
multiplicity sequence.

Finite generation changes the picture once the decreasing condition
is imposed. Example~\ref{ex:periodic-zero} first shows that finite
generation alone is not enough: a Noetherian graded family that is not
a filtration may already fail to have a convergent zeroth component.
For a Noetherian graded filtration, however,
Theorem~\ref{thm:noetherian-convergence} proves that every asymptotic
component exists and is finite. More precisely, there is a standard
Veronese index $q\geq1$ such that
\[
c_i(\mathcal I;R)
=
\frac{c_i(I_q;R)}{q^i}
\qquad(0\leq i\leq d).
\]
In particular, all asymptotic components are rational. The proof shows
that finite generation reduces each residue class modulo $q$ to powers
of $I_q$ multiplied by a fixed factor, while the fixed-factor
estimates developed in Section~\ref{sec:one-ideal} control the
resulting error terms. The same reduction gives a total multiplicity
formula: Proposition~\ref{prop:total-multiplicity-family} expresses
the multiplicity of the associated graded ring of the standard
Veronese in terms of the asymptotic sequence,
\[
e_{\mathfrak N_q}\bigl(G_q(\mathcal I)\bigr)
=
\sum_{i=0}^{d}q^ic_i(\mathcal I;R).
\]

The Noetherian theory also yields a family version of the Rees
criterion. For equidimensional universally catenary $R$ and nested
Noetherian graded filtrations
$\mathcal I\subseteq\mathcal J$,
Theorem~\ref{thm:family-Rees-criterion} shows that the following are
equivalent: $\mathscr R(\mathcal J)$ is integral over
$\mathscr R(\mathcal I)$; one has
$c_i(\mathcal I;R)\leq c_i(\mathcal J;R)$ for every $i$; the two
asymptotic multiplicity sequences are equal; and the two family
algebras have the same integral closure in $R[t]$. This extends the
fixed-ideal criterion of \cite{PTUV}. To clarify its valuative
content, we distinguish termwise integral closure, integral closure of
the family algebra, and divisorial saturation. For Noetherian
filtrations in an excellent local domain, algebraic closure and
divisorial saturation coincide, and
Corollary~\ref{thm:noeth-valuative-rees} gives
\[
\cvec(\mathcal I;R)=\cvec(\mathcal J;R)
\quad\Longleftrightarrow\quad
\widetilde{\mathcal I}^{\,\mathrm{div}}
=
\widetilde{\mathcal J}^{\,\mathrm{div}}.
\]
The corresponding equality may fail for non-Noetherian families,
showing again that finite generation is essential in this comparison.
For $\mathfrak m$-primary families, the resulting valuative picture is
consistent with Cutkosky's asymptotic Rees theorem and does not require
finite generation.

A second mechanism for convergence is geometric and does not require
Noetherianity of the family algebra. Let $(X,0)$ be a reduced complex
analytic germ of pure dimension and let $R=\mathcal O_{X,0}$.
If a single proper modification $\pi:Y\to X$ makes every
$I_n\mathcal O_Y$ invertible, then
Theorem~\ref{thm:common-analytic-model-convergence} proves the
existence and finiteness of all positive asymptotic components. The
common model places the corresponding divisors in a fixed
finite-dimensional real vector space; their coefficients are
subadditive, so Fekete's lemma produces limiting divisors, while
intersection theory converts these limits into the asymptotic
multiplicity components. This viewpoint is related to the recent work
of Cutkosky and Monta\~no \cite{CutMontanoFamilies}, who obtain
intersection-theoretic formulas for asymptotic multiplicities and
degree functions of graded $\mathfrak m$-primary families. Our result
proceeds in a different direction by treating all positive components
of the Achilles--Manaresi sequence for arbitrary ideals under a
common-model hypothesis.

The final part of the paper develops a diagonal mixed theory.
Callejas--Bedregal and Jorge P\'erez \cite{CBJP} introduced mixed
multiplicity-sequence coefficients for several arbitrary ideals.
Using their diagonal multigraded construction with codimension
indexing, we write
\[
c^\Delta_{i,\boldsymbol\alpha}
(J_1|\cdots|J_s;R),
\qquad |\boldsymbol\alpha|=i-1.
\]
For $s=1$ these reduce to the ordinary components,
$c^\Delta_{i,(i-1)}(J;R)=c_i(J;R)$. In the
$\mathfrak m$-primary case, Theorem~\ref{thm:mixed-primary} proves
that all diagonal asymptotic components exist without any
finite-generation assumption, vanish for $i<d$, and for
$|\boldsymbol\alpha|=d-1$ satisfy
\[
c^\Delta_{d,\boldsymbol\alpha}
(\mathcal I^{(1)}|\cdots|\mathcal I^{(s)};R)
=
\sum_{j=1}^s
e\bigl(
(\mathcal I^{(1)})^{[\alpha_1]},\ldots,
(\mathcal I^{(j)})^{[\alpha_j+1]},\ldots,
(\mathcal I^{(s)})^{[\alpha_s]};R
\bigr).
\]
Thus a diagonal coefficient is generally a sum of adjacent
Rees--Teissier mixed multiplicities rather than a single mixed
multiplicity.

For nonprimary Noetherian filtrations a new issue appears.
Simultaneous homogeneity controls the standard Veronese subsequence,
but the remaining residue classes require a mixed fixed-factor
estimate. Proposition~\ref{thm:mixed-noetherian-convergence} proves
convergence under this estimate, while
Problem~\ref{prob:mixed-fixed-factor} isolates its validity in general
as a concrete open question. Whenever the diagonal mixed components
exist, Proposition~\ref{thm:asymptotic-product} expresses the
asymptotic component of the termwise product family as the
corresponding multinomial sum. In particular, the ordinary
multiplicity components of the product of Noetherian filtrations
exist independently of the unresolved separation of the individual
mixed summands.

In the complex analytic setting, the common-model method resolves the
mixed convergence problem under simultaneous principalization.
Lemma~\ref{lem:mixed-intersection-formula} derives an intersection
formula directly from the diagonal multigraded quotient, retaining
the local intersection data on the proper fiber.
Theorem~\ref{thm:mixed-common-model-convergence} then proves that all
diagonal mixed components converge whenever the families are
simultaneously principalized on a single proper modification. Taking
$s=1$ recovers
Theorem~\ref{thm:common-analytic-model-convergence}, so the mixed
theorem genuinely extends the one-family common-model result.
Problem~\ref{prob:common-model} asks for an algebraic counterpart over
excellent local domains and isolates the local intersection-theoretic
input still missing from the general case.

The paper is organized as follows.
Section~\ref{sec:one-ideal} recalls the multiplicity sequence and
establishes the homogeneity, additivity, torsion-reduction, and
fixed-factor results used throughout the paper.
Section~\ref{sec:definition} introduces the asymptotic multiplicity
sequence, determines its first possible components, constructs the
surface examples with arbitrary higher behavior, and proves the
Noetherian and analytic common-model convergence theorems.
Section~\ref{sec:closure} compares termwise, algebraic, and divisorial
closures and proves the family Rees and valuative criteria.
Finally, Section~\ref{sec:mixed-noetherian-filtrations} develops the
diagonal mixed theory, identifies the primary case with Cutkosky's
mixed multiplicities, isolates the mixed fixed-factor problem, studies
termwise products, and proves the mixed intersection and common-model
convergence theorems.

\section{Preliminaries on multiplicity sequences}
\label{sec:one-ideal}

Throughout the paper, unless otherwise stated, $(R,\m,k)$ is a
Noetherian local ring of dimension $d\geq1$. In this section,
$I\subsetneq R$ is an ideal.
We recall only the notation used later.  For details concerning these topics, see \cite{AM97,AM06,PTUV}.

\subsection*{The multiplicity sequence}

Set
\[
 B=\mathcal R(I)=R[It],\qquad
 G=\gr_I(R)=\bigoplus_{v\geq0}I^v/I^{v+1},
 \qquad
 \ell(I)=\dim(B/\m B).
\]
The $(u,v)$-component of
$\gr_{\m}(G)$ is
\[
 \frac{\m^uI^v+I^{v+1}}{\m^{u+1}I^v+I^{v+1}}.
\]
The double cumulative function is
$H_I(r,s)=\sum_{v=0}^s
\lambda_R(I^v/(\m^{r+1}I^v+I^{v+1}))$.
For $r,s\gg0$ it agrees with a polynomial whose homogeneous part
of total degree $d$ is
\begin{equation}\label{eq:multiplicity-sequence-polynomial}
 \sum_{i=0}^d
 \frac{c_i(I;R)}{(d-i)!i!}\,r^{d-i}s^i.
\end{equation}

\begin{definition}\label{def:multiplicity-sequence}
The sequence \(\cvec(I;R)=(c_0(I;R),\ldots,c_d(I;R))\)
defined by \eqref{eq:multiplicity-sequence-polynomial} is called the
\emph{multiplicity sequence} of $I$.
\end{definition}

If $c_j^{\rm AM}(I;R)$ denotes the ordinary dimension indexing of
Achilles--Manaresi, then
\begin{equation}\label{eq:indexing-AM-PTUV}
 c_i(I;R)=c_{d-i}^{\rm AM}(I;R).
\end{equation}
Thus our index records codimension.  This is the convention of
\cite[Section~2]{PTUV}.

\begin{proposition}\label{prop:classical}
The coefficients $c_i(I;R)$ are nonnegative integers. Furthermore:
\begin{enumerate}[label=\textup{(\roman*)}]
    \item $c_i(I;R) = 0$ for $i < d - \dim(R/I)$ or $i > \ell(I)$.
    \item $c_0(I;R) = \sum_{\substack{\p \supseteq I \\ \dim(R/\p) = d}} \lambda_{R_{\p}}(R_{\p}) e(R/\p)$.
    \item $c_d(I;R) = j(I;R)$.
    \item If $I$ is $\m$-primary, then $\cvec(I;R) = (0, \ldots, 0, e(I;R))$.
\end{enumerate}
\end{proposition}

Notice that
$c_0$ is governed by the top-dimensional components contained in
$V(I)$, whereas $c_d$ is the $j$-multiplicity. The first possible positive component has the following associativity
formula. Assume that $R$ is equidimensional and catenary, and put
$h=\operatorname{ht}I>0$.  Then, \cite[Proposition~2.4(a)]{PTUV} gives that
$$ c_h(I;R)=
 \sum_{\substack{\p\in\Min(R/I)\\ \operatorname{ht}\p=h}}
 e(IR_{\p};R_{\p})e(R/\p).$$

We collect here the basic properties and comparison results for multiplicity sequences that will be used throughout the paper. Besides recalling the classical one-ideal theory, we establish the homogeneity, additivity, torsion-reduction, and fixed-factor estimates that provide the main technical tools for the asymptotic arguments developed in the following sections.

\begin{lemma}\label{lem:additivity}
Let $A$ be a standard graded Noetherian $R$-algebra and \(0\longrightarrow N'\longrightarrow N\longrightarrow N'' \longrightarrow0\)
an exact sequence of finite graded $A$-modules.  If the cumulative
Hilbert polynomials \(\mathscr H_M(r,s)= \sum_{v=0}^{s}\lambda_R(M_v/\m^{r+1}M_v)\)
have total degree at most $D$, then their homogeneous parts of degree
$D$ are additive.
\end{lemma}

\begin{proof}
For every $r\geq0$, reduction modulo $\mathfrak m^{r+1}$ gives an exact
sequence
\[
0\longrightarrow
\frac{N'}{N'\cap\mathfrak m^{r+1}N}
\longrightarrow
\frac{N}{\mathfrak m^{r+1}N}
\longrightarrow
\frac{N''}{\mathfrak m^{r+1}N''}
\longrightarrow0.
\]
Since $\mathfrak m$ has degree zero, this sequence is graded. Summing
the lengths of its components up to degree $s$, we obtain
$\mathscr H_N(r,s)=\mathscr H_{N''}(r,s)+Q(r,s)$, where
$Q(r,s)=\sum_{v=0}^s
\lambda_R\bigl(N'_v/(N'_v\cap\mathfrak m^{r+1}N_v)\bigr)$.

By Artin--Rees, there exists $c\geq0$ such that, for $r\gg0$,
$\mathfrak m^{r+1}N'\subseteq
N'\cap\mathfrak m^{r+1}N\subseteq
\mathfrak m^{r+1-c}N'$. Hence
$\mathscr H_{N'}(r-c,s)\leq Q(r,s)\leq\mathscr H_{N'}(r,s)$.
A fixed translation in the $r$-variable does not change the
homogeneous part of total degree $D$. Dividing these inequalities by
$t^D$ after replacing $(r,s)$ by $(tr,ts)$ and letting $t\to\infty$
shows that $Q$ and $\mathscr H_{N'}$ have the same homogeneous part of
degree $D$. Therefore the degree-$D$ homogeneous parts satisfy $[\mathscr H_N]_D
=
[\mathscr H_{N'}]_D+[\mathscr H_{N''}]_D,$
as required.
\end{proof}

\begin{proposition}
\label{prop:homogeneity}
For every $a\geq1$ and $0\leq i\leq d$,
\begin{equation}\label{eq:homogeneity}
 c_i(I^a;R)=a^ic_i(I;R).
\end{equation}

\end{proposition}

\begin{proof}
We first consider $1\leq i\leq d$. Set $h_I(r,u)=\lambda_R\bigl(I^u/(\mathfrak m^{r+1}I^u+I^{u+1})\bigr)$. Since $H_I(r,s)=\sum_{u=0}^s h_I(r,u)$, the homogeneous part of total degree $d-1$ of the eventual polynomial associated with $h_I(r,u)$ is $\sum_{i=1}^d \frac{c_i(I;R)}{(d-i)!(i-1)!}r^{d-i}u^{i-1}$. For each $v\geq0$, the quotient $I^{av}/I^{a(v+1)}$ has a filtration with successive quotients $I^{av+j}/I^{av+j+1}$, $0\leq j<a$. By Lemma~\ref{lem:additivity}, the top homogeneous parts of the corresponding $\mathfrak m$-cumulative Hilbert polynomials are additive. For fixed $i$ and $j$, one has $\sum_{v=0}^s(av+j)^{i-1}=\frac{a^{i-1}}{i}s^i+O(s^{i-1})$. Hence each successive quotient contributes $\frac{a^{i-1}c_i(I;R)}{(d-i)!i!}r^{d-i}s^i$ to the degree-$d$ part, and summing over the $a$ quotients gives $\frac{a^ic_i(I;R)}{(d-i)!i!}r^{d-i}s^i$. Therefore $c_i(I^a;R)=a^ic_i(I;R)$ for $1\leq i\leq d$. For $i=0$, Proposition~\ref{prop:classical} gives
$c_0(I;R)=\sum_{\mathfrak p\supseteq I,\ \dim(R/\mathfrak p)=d}
\lambda_{R_{\mathfrak p}}(R_{\mathfrak p})e(R/\mathfrak p)$.
Since a prime ideal contains $I^a$ if and only if it contains $I$, the
same sum computes $c_0(I^a;R)$. Hence
$c_0(I^a;R)=c_0(I;R)=a^0c_0(I;R)$..
\end{proof}




\begin{lemma}
\label{lem:positive-fixed-factor}
Let $J,K\subsetneq R$ be ideals of positive height. For every
$1\leq i\leq d$,
\begin{equation}
\label{eq:fixed-factor}
c_i(J^nK;R)
=
n^ic_i(J;R)+O(n^{i-1})
\qquad(n\longrightarrow\infty).
\end{equation}
In particular, \(c_i(J^nK;R)=O(n^i)\).
\end{lemma}

\begin{proof}
Set $A=R[JT_1,KT_2]$ and let the finite bigraded $A$-modules
$U=\bigoplus_{u,v\geq0}J^uK^v/J^{u+1}K^v$ and
$V=\bigoplus_{u,v\geq0}J^uK^v/J^uK^{v+1}$. Their
$\mathfrak m$-cumulative Hilbert functions have eventual polynomials
of total degree at most $d-1$; let $P_J(r,u,v)$ and $P_K(r,u,v)$
denote their homogeneous parts of degree $d-1$. The usual
Artin--Rees argument gives additivity of these top homogeneous parts
along finite filtrations.

Fix $n\geq1$. For each $s$, filter
$J^{ns}K^s/J^{n(s+1)}K^{s+1}$ first through the $n$ successive
$J$-increments and then through one $K$-increment. The successive
quotients are the components of $U$ in degrees $(ns+a,s)$,
$0\leq a<n$, and the component of $V$ in degree $(ns+n,s)$.
Hence the degree-$(d-1)$ part of the corresponding
$\mathfrak m$-cumulative Hilbert polynomial is obtained from
$\sum_{a=0}^{n-1}P_J(r,ns+a,s)+P_K(r,ns+n,s)$. Since translation
in the $u$-variable changes only terms of smaller total degree in
$(r,s)$, its degree-$(d-1)$ part is
\begin{equation}\label{eq:fixed-factor-polynomials}
nP_J(r,ns,s)+P_K(r,ns,s).
\end{equation}

The coefficient of $r^{d-i}s^{i-1}$ in
\eqref{eq:fixed-factor-polynomials} is a polynomial in $n$ of degree
at most $i$. In the first term, a monomial
$r^{d-i}u^qv^{i-1-q}$ contributes a multiple of $n^{q+1}$, while in
the second term it contributes a multiple of $n^q$. Consequently,
the coefficient of $n^i$ can only come from the monomial
$r^{d-i}u^{i-1}$ in $P_J$. To identify this coefficient, fix $v\gg0$. The $u$-graded slice of
$U$ is $\operatorname{gr}_J(K^v)$, so the coefficient of
$r^{d-i}u^{i-1}$ in $P_J$ is
$c_i(J;K^v)/((d-i)!(i-1)!)$. Since $\operatorname{ht}K>0$, one has
$\dim R/K^v<d$. Additivity of the multiplicity sequence in
$0\to K^v\to R\to R/K^v\to0$ therefore yields
$c_i(J;K^v)=c_i(J;R)$. Comparing the coefficient of
$r^{d-i}s^{i-1}$ gives
$c_i(J^nK;R)=n^ic_i(J;R)+O(n^{i-1})$. The final assertion follows immediately.
\end{proof}


\begin{lemma}
\label{lem:torsion-reduction}
Let $L\subsetneq R$ be an ideal, set
\(
H_L=0:_RL^\infty,
\,\,\,
\overline R=R/H_L,
\,\,\,
\overline L=L\overline R,
\)
and put $e=\dim\overline R$. For every $1\leq i\leq d$, one has
\[
c_i(L;R)
=
\begin{cases}
c_i(\overline L;\overline R),& e=d,\\
0,& e<d,
\end{cases}
\]
where we use the convention $\dim 0=-\infty$. In the first case,
\(
0:_{\overline R}\overline L^\infty=0,
\)
hence
\(
\operatorname{grade}_{\overline R}\overline L>0.
\)
\end{lemma}

\begin{proof}
Choose $N$ such that $H_L=0:_RL^N$ and set
$F_v=H_L\cap L^v$. By Artin--Rees, there exists $a$ such that
$F_v=L^{v-a}F_a$ for every $v\geq a$. Since $F_a\subseteq H_L$ and
$L^NH_L=0$, we have $F_v=0$ for every $v\geq a+N$. For each $v$, the natural map
$L^v/L^{v+1}\to(\overline L)^v/(\overline L)^{v+1}$ has kernel
$F_v/F_{v+1}$. Hence there is an exact sequence of finite graded
modules
\begin{equation}
\label{eq:graded-torsion-sequence}
0\longrightarrow
\bigoplus_{v\geq0}F_v/F_{v+1}
\longrightarrow
\operatorname{gr}_L(R)
\longrightarrow
\operatorname{gr}_{\overline L}(\overline R)
\longrightarrow0.
\end{equation}

The left-hand module vanishes in all sufficiently large graded
degrees. Its cumulative Hilbert polynomial is therefore eventually
independent of the $L$-degree variable, so its degree-$d$ homogeneous
part can contribute only to the component with index $0$. On the
other hand, the cumulative Hilbert polynomial of
$\operatorname{gr}_{\overline L}(\overline R)$ has total degree at
most $e=\dim\overline R$. If $e<d$, Lemma~\ref{lem:additivity}, applied in total degree $d$ to
\eqref{eq:graded-torsion-sequence}, shows that the degree-$d$ part of
the cumulative polynomial of $\operatorname{gr}_L(R)$ has no term
involving a positive power of the $L$-degree variable. Thus
$c_i(L;R)=0$ for every $i\geq1$. If $e=d$, the same lemma gives
$c_i(L;R)=c_i(\overline L;\overline R)$ for every $1\leq i\leq d$.

It remains to prove the last assertion. If
$\overline x\,\overline L^s=0$, then $xL^s\subseteq H_L$. Since
$L^NH_L=0$, it follows that $xL^{s+N}=0$, and hence
$x\in0:_RL^{s+N}=H_L$. Therefore $\overline x=0$, so
$0:_{\overline R}\overline L^\infty=0$. In particular,
$0:_{\overline R}\overline L=0$, and therefore $\overline L$ is not
contained in any prime of $\operatorname{Ass}(\overline R)$. Since
$\operatorname{Ass}(\overline R)$ is finite, prime avoidance yields
an $\overline R$-regular element in $\overline L$. Hence
$\operatorname{grade}_{\overline R}\overline L>0$.
\end{proof}

\begin{lemma}
\label{lem:fixed-factor}
Let $J,K\subsetneq R$ be ideals such that
\(
\sqrt J=\sqrt K.
\)
Then, for every $1\leq i\leq d$,
\begin{equation}
\label{eq:fixed-radical-estimate}
c_i(J^nK;R)
=
n^ic_i(J;R)+O(n^{i-1}).
\end{equation}
\end{lemma}
\begin{proof}
The saturation $0:_RL^\infty$ depends only on $\sqrt L$. Indeed, if
$\sqrt A=\sqrt B$, then $A^r\subseteq B$ and $B^s\subseteq A$ for
some $r,s\geq1$, and these inclusions give
$0:_RA^\infty=0:_RB^\infty$. Since
$\sqrt{J^nK}=\sqrt J=\sqrt K$, we may therefore set
\[
H:=0:_RJ^\infty=0:_RK^\infty=0:_R(J^nK)^\infty
\]
for every $n\geq1$, and put $S=R/H$. If $\dim S<d$, Lemma~\ref{lem:torsion-reduction} gives
$c_i(J;R)=c_i(J^nK;R)=0$ for every $i\geq1$, so the assertion is
immediate. Assume now that $\dim S=d$. Applying
Lemma~\ref{lem:torsion-reduction} to $J$, $K$, and $J^nK$, we obtain
that $JS$ and $KS$ have positive grade, hence positive height, and
that
$c_i(J;R)=c_i(JS;S)$ and
$c_i(J^nK;R)=c_i((JS)^nKS;S)$. Now, Lemma~\ref{lem:positive-fixed-factor}, applied in $S$, now gives
$c_i((JS)^nKS;S)=n^ic_i(JS;S)+O(n^{i-1})$. Therefore
$c_i(J^nK;R)=n^ic_i(J;R)+O(n^{i-1})$, as required.
\end{proof}


\section{Asymptotic invariants: definitions, finiteness, and existence}
\label{sec:definition}

In this section, we define the asymptotic multiplicity-sequence components associated
with a graded family. We first
record their basic formal properties and then study finiteness and
existence. Examples are placed next to the results they illustrate.

\begin{definition}
\label{def:family}
A \emph{graded family} on $R$ is a sequence
\(
\mathcal I=\{I_n\}_{n\geq0}
\)
such that
\(
I_0=R
\,\,\,\text{and}\,\,\,
I_mI_n\subseteq I_{m+n}
\)
for all $m,n\geq0$. Throughout this paper, we require
$I_n\subsetneq R$ for $n\geq1$. The family is a \emph{graded filtration} if $I_{n+1}\subseteq I_n$  for $n\geq 1$.
It is called \emph{Noetherian} if its family algebra
\(
\mathscr R(\mathcal I)
=
\bigoplus_{n\geq0}I_nt^n
\subseteq R[t]
\)
is finitely generated over $R$.
\end{definition}

For a family that is not decreasing, the quotients $I_n/I_{n+1}$
do not necessarily form an associated graded algebra. The following
definitions avoid this difficulty by applying the classical
invariants separately to each ideal $I_n$.

\begin{definition}
\label{def:asymptotic}
Let $\mathcal I$ be a graded family. For every $0\leq i\leq d$, set
\begin{equation}
\label{eq:asymptotic}
\lowc_i(\mathcal I;R)
=
\liminf_{n\to\infty}
\frac{c_i(I_n;R)}{n^i},
\qquad
\upc_i(\mathcal I;R)
=
\limsup_{n\to\infty}
\frac{c_i(I_n;R)}{n^i}.
\end{equation}
These numbers belong to $[0,+\infty]$. If
\(
\lowc_i(\mathcal I;R)=\upc_i(\mathcal I;R),
\)
their common value is denoted by
\(
c_i(\mathcal I;R).
\)
If every component exists and is finite, then
\[
\cvec(\mathcal I;R)
=
\bigl(
c_0(\mathcal I;R),\ldots,c_d(\mathcal I;R)
\bigr)
\]
is called the \emph{finite asymptotic multiplicity sequence}. If every component exists in $[0,+\infty]$, allowing infinite
entries, it is called the \emph{extended asymptotic multiplicity
sequence}.
\end{definition}

No global limit is presupposed by these definitions. In fact, the
lower and upper values could be formed for an arbitrary sequence of
ideals. The graded condition is used to obtain the structural and
existence results below.


\begin{proposition}\label{prop:compatibility}
\begin{itemize}
    \item[(i)] For the adic family $\mathcal{I}=\{I^n\}_{n\geq0}$, $c_i(\mathcal{I};R)=c_i(I;R)$ for $0\leq i\leq d$.
    \item[(ii)] If $\mathcal{I}$ is $\mathfrak{m}$-primary, then $\cvec(\mathcal{I};R) = (0,\ldots,0,e(\mathcal{I};R))$, where $e(\mathcal{I};R) = \lim_{n\to\infty}\frac{e(I_n;R)}{n^d} < \infty$.

\end{itemize}
\end{proposition}

\begin{proof}
For (i), Proposition~\ref{prop:homogeneity} gives
$c_i(I^n;R)=n^ic_i(I;R)$ for every $n\geq1$. Hence
$c_i(\mathcal I;R)=c_i(I;R)$ for all $0\leq i\leq d$.

For (ii), since each $I_n$ is $\mathfrak m$-primary,
Proposition~\ref{prop:classical} gives
$c_i(I_n;R)=0$ for $i<d$ and
$c_d(I_n;R)=e(I_n;R)$. Therefore all asymptotic components below the
top one are zero. By \cite[Theorem~1.3]{Cut}, the limit
$e(\mathcal I;R)=\lim_{n\to\infty}e(I_n;R)/n^d$ exists. Moreover,
$I_1^n\subseteq I_n$ for every $n\geq1$, so monotonicity and
homogeneity of Hilbert--Samuel multiplicity give
$e(I_n;R)\leq e(I_1^n;R)=n^de(I_1;R)$. Hence
$e(\mathcal I;R)\leq e(I_1;R)<\infty$, and consequently
$\cvec(\mathcal I;R)=(0,\ldots,0,e(\mathcal I;R))$.
\end{proof}

\begin{example}\label{ex:reparametrized}
Let $J\subsetneq R$ and let $(a_n)_{n\geq1}$ be positive integers
satisfying $a_{m+n}\leq a_m+a_n$. Then $I_0=R$ and $I_n=J^{a_n}$
define a graded family. It is a filtration if $(a_n)$ is
nondecreasing. Fekete's lemma gives
$\alpha=\lim a_n/n=\inf a_n/n\in[0,a_1]$. By homogeneity,
\[
 c_0(\I;R)=c_0(J;R),\qquad
 c_i(\I;R)=\alpha^i c_i(J;R)\quad(1\leq i\leq d).
\]
For instance:
\begin{enumerate}[label=\textup{(\alph*)}]
\item If $\alpha>0$, then $a_n=\lceil\alpha n\rceil$ is
nondecreasing and subadditive. In $k[[x]]$ the resulting filtration
$(x^{\lceil\alpha n\rceil})$ has sequence $(0,\alpha)$, allowing
irrational components.
\item The choice $a_n=\lceil\sqrt n\rceil$ is also nondecreasing and
subadditive, but $a_n/n\to0$. Thus all positive components can vanish
for a filtration of nonzero proper ideals.
\item More generally, let $\K=\{K_n\}$ be a graded family in an
equidimensional universally catenary ring. If
$\overline{K_n}=\overline{J^{a_n}}$ for $n\gg0$ and
$a_n/n\to\alpha<\infty$, integral-closure invariance gives the same
formulas for $\K$; subadditivity of $(a_n)$ is unnecessary for this
last comparison.
\end{enumerate}
\end{example}


\begin{proposition}
\label{prop:radicals}
Let $\mathcal I$ be a graded family.
\begin{enumerate}[label=\textup{(\roman*)}]
\item For every $n\geq1$, \(I_1^n\subseteq I_n.\)
In addition, $\sqrt{I_1}\subseteq\sqrt{I_n}$ and  $V(I_n)\subseteq V(I_1).$

\item If \(h_0=d-\dim(R/I_1),\)
then $c_i(I_n;R)=0$  for $i<h_0$ and
for every $n\geq1$.

\item If $\ell(I_n)\leq s$, with $n\gg0$, 
then $\lowc_i(\mathcal I;R)
=
\upc_i(\mathcal I;R)
=
0$ for all $i>s.$

\item If $\mathcal I$ is a filtration, then $\sqrt{I_n}=\sqrt{I_1}$  for $n\geq1$,
and \(c_0(I_n;R)=c_0(I_1;R).\)
\end{enumerate}
\end{proposition}
\begin{proof}
For (i), the graded condition gives $I_1^n\subseteq I_n$ by induction.
Hence $\sqrt{I_1}\subseteq\sqrt{I_n}$, and therefore
$V(I_n)\subseteq V(I_1)$.

Since $I_1^n\subseteq I_n$, one has
$\dim(R/I_n)\leq\dim(R/I_1)$. Thus
$d-\dim(R/I_n)\geq h_0$, and Proposition~\ref{prop:classical}\textup{(i)}
gives $c_i(I_n;R)=0$ for every $i<h_0$, proving (ii).

For (iii), if $\ell(I_n)\leq s$ for all $n\gg0$, then
Proposition~\ref{prop:classical}\textup{(i)} gives
$c_i(I_n;R)=0$ for every $i>s$ and all sufficiently large $n$.
Hence $\lowc_i(\mathcal I;R)=\upc_i(\mathcal I;R)=0$ for $i>s$.

For (iv), if $\mathcal I$ is a filtration, then
$I_1^n\subseteq I_n\subseteq I_1$ for every $n\geq1$. Taking radicals
gives $\sqrt{I_n}=\sqrt{I_1}$. Therefore the sets of
top-dimensional primes containing $I_n$ and $I_1$ coincide, and
Proposition~\ref{prop:classical}\textup{(ii)} yields
$c_0(I_n;R)=c_0(I_1;R)$.
\end{proof}

\begin{proposition}
\label{prop:operations}
Fix an integer $q\geq1$. The families
$\mathcal I^{[q]}=\{I_n^q\}_{n\geq0}$ and
$\mathcal I^{(q)}=\{I_{qn}\}_{n\geq0}$ are graded, and they are
filtrations whenever $\mathcal I$ is a filtration. Moreover, for every
$0\leq i\leq d$,
\[
\lowc_i(\mathcal I^{[q]};R)=q^i\lowc_i(\mathcal I;R),
\qquad
\upc_i(\mathcal I^{[q]};R)=q^i\upc_i(\mathcal I;R).
\]
For the Veronese reindexing,
\begin{equation}
\label{eq:veronese-bounds}
q^i\lowc_i(\mathcal I;R)
\leq
\lowc_i(\mathcal I^{(q)};R)
\leq
\upc_i(\mathcal I^{(q)};R)
\leq
q^i\upc_i(\mathcal I;R).
\end{equation}
If $c_i(\mathcal I;R)$ exists, then
$c_i(\mathcal I^{(q)};R)=q^ic_i(\mathcal I;R)$.
\end{proposition}

\begin{proof}
The graded and decreasing conditions follow directly from those of
$\mathcal I$. For termwise powers, Proposition~\ref{prop:homogeneity}
gives
$c_i(I_n^q;R)/n^i=q^ic_i(I_n;R)/n^i$, and taking lower and upper
limits yields the asserted equalities. For the Veronese family,
$c_i(I_{qn};R)/n^i=q^i\,c_i(I_{qn};R)/(qn)^i$. Since
$\{c_i(I_{qn};R)/(qn)^i\}_{n\geq1}$ is a subsequence of
$\{c_i(I_n;R)/n^i\}_{n\geq1}$, the standard inequalities for
$\liminf$ and $\limsup$ give \eqref{eq:veronese-bounds}. If the latter
sequence converges, every subsequence has the same limit, and therefore
$c_i(\mathcal I^{(q)};R)=q^ic_i(\mathcal I;R)$.
\end{proof}


We first determine the asymptotic behavior of the initial components, without assuming the existence of the entire asymptotic multiplicity sequence. The zeroth component is governed by the top-dimensional components of $R$, while the first possible positive component will follow from the classical associativity formula (see \cite{AM97,PTUV}). No finite-generation hypothesis is needed. Write $\operatorname{Assh}(R)=\{\mathfrak p\in\operatorname{Spec}R:\dim R/\mathfrak p=d\}$, the finite set of top-dimensional minimal primes of $R$.

\begin{proposition}
\label{prop:zeroth-component}
Let \((R,\mathfrak m)\) be a \(d\)-dimensional Noetherian local ring
and let \(\mathcal I=\{I_n\}_{n\geq0}\) be a graded family of proper
ideals. For each \(\mathfrak p\in\operatorname{Assh}(R)\), set
\(
\Gamma_{\mathfrak p}(\mathcal I)
=
\{n\geq1:I_n\nsubseteq\mathfrak p\}.
\)
Then \(c_0(\mathcal I;R)\) exists if and only if
\(
\gcd\Gamma_{\mathfrak p}(\mathcal I)=1
\)
for every \(\mathfrak p\in\operatorname{Assh}(R)\) such that
\(\Gamma_{\mathfrak p}(\mathcal I)\neq\varnothing\).
\end{proposition}

\begin{proof}
By Proposition~\ref{prop:classical} \textup{(ii)}, for every $n\geq1$,
\begin{equation}
\label{eq:c0-associativity}
c_0(I_n;R)
=
\sum_{\substack{\mathfrak p\in\operatorname{Assh}(R)\\
I_n\subseteq\mathfrak p}}
\lambda_{R_{\mathfrak p}}(R_{\mathfrak p})e(R/\mathfrak p).
\end{equation}
Fix $\mathfrak p\in\operatorname{Assh}(R)$. If
$m,n\in\Gamma_{\mathfrak p}(\mathcal I)$, choose
$a\in I_m\setminus\mathfrak p$ and $b\in I_n\setminus\mathfrak p$.
Since $I_mI_n\subseteq I_{m+n}$ and $\mathfrak p$ is prime,
$ab\in I_{m+n}\setminus\mathfrak p$. Thus
$\Gamma_{\mathfrak p}(\mathcal I)$ is an additive subsemigroup of
$\mathbb Z_{>0}$. If it is nonempty and
$g_{\mathfrak p}=\gcd\Gamma_{\mathfrak p}(\mathcal I)$, then the
standard structure of additive subsemigroups of $\mathbb Z_{>0}$
shows that $\Gamma_{\mathfrak p}(\mathcal I)$ contains every
sufficiently large multiple of $g_{\mathfrak p}$ and no integer not
divisible by $g_{\mathfrak p}$. Hence, for $n\gg0$,
$I_n\nsubseteq\mathfrak p$ if and only if
$g_{\mathfrak p}\mid n$.

Suppose first that every nonempty
$\Gamma_{\mathfrak p}(\mathcal I)$ has gcd $1$. Then, for each
$\mathfrak p\in\operatorname{Assh}(R)$, either
$\Gamma_{\mathfrak p}(\mathcal I)=\varnothing$, in which case
$I_n\subseteq\mathfrak p$ for every $n$, or
$I_n\nsubseteq\mathfrak p$ for all sufficiently large $n$.
Since $\operatorname{Assh}(R)$ is finite,
\eqref{eq:c0-associativity} is eventually constant. Thus
$c_0(\mathcal I;R)$ exists.

Conversely, let $\Sigma$ be the nonempty set of primes
$\mathfrak p\in\operatorname{Assh}(R)$ for which
$g_{\mathfrak p}>1$, and let $N$ be the least common multiple of
these integers. For all sufficiently large $t$, every
$\mathfrak p\in\Sigma$ satisfies $tN\in\Gamma_{\mathfrak p}(\mathcal I)$
and $tN+1\notin\Gamma_{\mathfrak p}(\mathcal I)$. Primes with gcd $1$
do not contribute to \eqref{eq:c0-associativity} for either index once
$t$ is large, while primes with empty semigroup contribute to both.
Therefore, for $t\gg0$,
\[
c_0(I_{tN+1};R)-c_0(I_{tN};R)
=
\sum_{\mathfrak p\in\Sigma}
\lambda_{R_{\mathfrak p}}(R_{\mathfrak p})e(R/\mathfrak p)>0.
\]
Hence $c_0(I_n;R)$ cannot converge.
\end{proof}

\begin{remark}
Note that Formula~\eqref{eq:c0-associativity} and the proof above show that
$c_0(I_n;R)$ is eventually periodic. More precisely, an eventual
period is given by the least common multiple of the nontrivial integers
$g_{\mathfrak p}=\gcd\Gamma_{\mathfrak p}(\mathcal I)$. Moreover,
$0\leq c_0(I_n;R)\leq e(R)$ for every $n$. If $R$ is a domain and
$I_1\neq0$, then $I_1^n\subseteq I_n$ implies $I_n\neq0$ for every
$n\geq1$. Since $\operatorname{Assh}(R)=\{(0)\}$, formula
\eqref{eq:c0-associativity} gives $c_0(I_n;R)=0$ for every $n$, even
without assuming that $\mathcal I$ is a filtration.
\end{remark}

\begin{example}
\label{ex:periodic-zero}
Let $R=k[[x]]$, set $I_0=R$, and define
$I_{2q}=(x^{2q})$ for $q\geq1$ and $I_{2q+1}=0$ for $q\geq0$.
Then $\mathcal I=\{I_n\}_{n\geq0}$ is a graded family and
$\mathscr R(\mathcal I)=R[x^2t^2]$, so in particular it is
Noetherian. It is not a filtration, since $I_2\nsubseteq I_1$. Since $R$ is a one-dimensional domain,
$\operatorname{Assh}(R)=\{(0)\}$. For every positive even $n$,
$I_n\neq0$, and hence Proposition~\ref{prop:classical}\textup{(ii)}
gives $c_0(I_n;R)=0$. For odd $n$, one has $I_n=0$, and the same
formula gives $c_0(I_n;R)=e(R)=1$. Therefore
$\lowc_0(\mathcal I;R)=0$ and $\upc_0(\mathcal I;R)=1$.
Thus finite generation of the family algebra alone does not guarantee
convergence of the zeroth asymptotic component.
\end{example}

The next result shows the first positive component that is not forced to
vanish. The argument reduces it, through the associativity formula, to
Hilbert--Samuel multiplicities of localized graded primary families.

\begin{theorem}
\label{thm:height}
Let \(R\) be an equidimensional catenary Noetherian local ring, and
let \(\mathcal I=\{I_n\}_{n\geq0}\) be a graded family of proper
ideals. If
\(
h:=\operatorname{ht}I_1>0,
\)
then
\(
c_i(\mathcal I;R)=0\) for 
$i<h$,
and \(c_h(\mathcal I;R)\) exists and is finite. Moreover,
\(
0\leq c_h(\mathcal I;R)\leq c_h(I_1;R).
\)
\end{theorem}

\begin{proof}
Since $I_1^n\subseteq I_n$, one has $\operatorname{ht}I_n\geq h$.
As $R$ is equidimensional and catenary,
$d-\dim(R/I_n)=\operatorname{ht}I_n$. Hence
Proposition~\ref{prop:classical} \textup{(i)} gives
$c_i(I_n;R)=0$ for every $i<h$, and therefore
$c_i(\mathcal I;R)=0$ for $i<h$.
Set
$\Lambda=\{\mathfrak p\in\operatorname{Min}(R/I_1):
\operatorname{ht}\mathfrak p=h\}$. Since $V(I_n)\subseteq V(I_1)$,
every height-$h$ prime containing $I_n$ belongs to $\Lambda$.
Thus the associativity formula
\cite[Proposition~2.4(a)]{PTUV} gives
\[
c_h(I_n;R)
=
\sum_{\mathfrak p\in\Lambda}
b_{\mathfrak p,n}e(R/\mathfrak p),
\]
where $b_{\mathfrak p,n}=e(I_nR_{\mathfrak p};R_{\mathfrak p})$ if
$I_nR_{\mathfrak p}\neq R_{\mathfrak p}$, and
$b_{\mathfrak p,n}=0$ otherwise. 

Fix $\mathfrak p\in\Lambda$, set $S=R_{\mathfrak p}$, and put
$J_n=I_nS$. Then $\dim S=h$, $J_1$ is
$\mathfrak pS$-primary, and $J_1^n\subseteq J_n$. Hence every proper
$J_n$ is $\mathfrak pS$-primary. If all $J_n$ are proper, then $\{J_n\}$ is a graded family of
$\mathfrak pS$-primary ideals, and \cite[Theorem~1.3]{Cut} gives the
finite limit
$\lim_{n\to\infty}b_{\mathfrak p,n}/n^h$.

Suppose instead that $J_q=S$ for some $q\geq1$. For every sufficiently
large $n$, write $n=kq+r$ with $1\leq r\leq q$. The graded condition
gives $J_q^kJ_r\subseteq J_n$, hence $J_r\subseteq J_n$. If $J_n$ is
proper, monotonicity of Hilbert--Samuel multiplicity gives
$b_{\mathfrak p,n}\leq e(J_r;S)$; if $J_n=S$, then
$b_{\mathfrak p,n}=0$. Since only the finitely many ideals
$J_1,\ldots,J_q$ occur, the sequence $b_{\mathfrak p,n}$ is bounded.
As $h>0$, it follows that
$\lim_{n\to\infty}b_{\mathfrak p,n}/n^h=0$. Thus every local normalized sequence has a finite limit. Since
$\Lambda$ is finite,
\begin{equation}
\label{eq:height-assoc}
c_h(\mathcal I;R)
=
\sum_{\mathfrak p\in\Lambda}
\left(
\lim_{n\to\infty}\frac{b_{\mathfrak p,n}}{n^h}
\right)e(R/\mathfrak p),
\end{equation}
so $c_h(\mathcal I;R)$ exists and is finite. Hence, $J_1^n\subseteq J_n$ gives
$0\leq b_{\mathfrak p,n}\leq e(J_1^n;S)
=n^he(J_1;S)$. Dividing by $n^h$, passing to the limit in
\eqref{eq:height-assoc}, and using the associativity formula for
$c_h(I_1;R)$ yields
$0\leq c_h(\mathcal I;R)\leq c_h(I_1;R)$.
\end{proof}

\begin{corollary}
\label{cor:fixed-radical-equimultiple}
Assume that $R$ is equidimensional and catenary, that
$h:=\operatorname{ht}I_1>0$, and that
$\sqrt{I_n}=\sqrt{I_1}$ for every $n\geq1$. Then
$c_h(\mathcal I;R)$ exists, is finite, and satisfies
\begin{equation}
\label{eq:fixed-radical}
c_h(\mathcal I;R)
=
\sum_{\substack{
\mathfrak p\in\operatorname{Min}(R/I_1)\\
\operatorname{ht}\mathfrak p=h}}
e(\mathcal I R_{\mathfrak p};R_{\mathfrak p})
e(R/\mathfrak p).
\end{equation}
If, in addition, $\ell(I_n)=h$ for all $n\gg0$, then all asymptotic
components exist and are finite, and
$c_i(\mathcal I;R)=0$ for every $i\neq h$.
\end{corollary}

\begin{proof}
For every
$\mathfrak p\in\operatorname{Min}(R/I_1)$ with
$\operatorname{ht}\mathfrak p=h$, the equality
$\sqrt{I_n}=\sqrt{I_1}$ implies
$I_nR_{\mathfrak p}\neq R_{\mathfrak p}$ for every $n\geq1$.
Thus the localized family
$\mathcal I R_{\mathfrak p}$ is a graded family of
$\mathfrak pR_{\mathfrak p}$-primary ideals. Hence the local limits
appearing in Theorem~\ref{thm:height} are precisely
$e(\mathcal I R_{\mathfrak p};R_{\mathfrak p})$, and
\eqref{eq:fixed-radical} follows from \eqref{eq:height-assoc}.

Theorem~\ref{thm:height} already gives
$c_i(\mathcal I;R)=0$ for $i<h$. If moreover
$\ell(I_n)=h$ for all $n\gg0$, then
Proposition~\ref{prop:classical}\textup{(i)} gives
$c_i(I_n;R)=0$ for every $i>h$ and all sufficiently large $n$.
Thus $c_i(\mathcal I;R)=0$ for $i\neq h$, while the existence and
finiteness of $c_h(\mathcal I;R)$ were proved above.
\end{proof}

The preceding existence theorem is sharp with respect to the index: already on a regular surface, the next component can exhibit arbitrary asymptotic behavior.

\begin{example}
\label{ex:surface-family}
Let $R=k[[x,y]]$, with $k$ infinite, and let $a,b\geq1$. Set
$L=x^a(x,y^b)$. Then $L$ is integrally closed and
\begin{equation}\label{eq:surface-single}
\sqrt L=(x),\qquad \htop L=1,\qquad
\ell(L)=2,\qquad \cvec(L;R)=(0,a,b(a+1)).
\end{equation}

Indeed, $(x,y^b)$ is integrally closed. If $z$ is integral over
$(x,y^b)$, then modulo $x$ its image is integral over $(y^b)$ in the
DVR $k[[y]]$, hence belongs to $(y^b)$. Thus
$z\in(x,y^b)$. Multiplication by $x^a$ preserves integral closedness:
if $z$ is integral over $x^a(x,y^b)$, then $z/x^a$ is integral over
$(x,y^b)$ and, in particular, over $R$. Since $R$ is normal,
$z/x^a\in R$, and hence $z/x^a\in(x,y^b)$.

Since $R$ is a domain and $L\neq0$, one has $c_0(L;R)=0$. Moreover,
$LR_{(x)}=(x^a)R_{(x)}$, so the associativity formula gives
$c_1(L;R)=a$. To compute $c_2$, choose
\[
f=x^a(\lambda x+\mu y^b),\qquad
g=x^a(\lambda' x+\mu' y^b),
\quad
\lambda\mu(\lambda\mu'-\lambda'\mu)\neq0,
\]
and write $l=\lambda x+\mu y^b$. Then
$(f):L^\infty=(l)$. Indeed, $(l)L\subseteq(f)$, so
$(l)\subseteq(f):L^\infty$. Conversely, if
$h\in(f):L^\infty$, then $hL^t\subseteq(l)$ for some $t$. Since
$R/(l)\cong k[[y]]$ is a domain and the image of $L$ is nonzero, this
forces $h\in(l)$. Modulo $l$, the image of $g$ is a nonzero scalar multiple of
$y^{b(a+1)}$. Hence \cite[Proposition~2.1]{PTUV} gives
$c_2(L;R)=\lambda_R(R/(l,g))=b(a+1)$. Since $c_2(L;R)>0$,
Proposition~\ref{prop:classical}\textup{(i)} gives $\ell(L)\geq2$,
while $\ell(L)\leq\dim R=2$. Thus $\ell(L)=2$, proving
\eqref{eq:surface-single}.

Now choose arbitrary positive integers $b_n$, set $I_0=R$, and define
$I_n=x^{n+1}(x,y^{b_n})$ for $n\geq1$. Then
\[
I_{n+1}\subseteq(x^{n+2})\subseteq I_n,\qquad
I_mI_n\subseteq(x^{m+n+2})\subseteq I_{m+n}
\quad(m,n\geq1),
\]
so $\mathcal I$ is a graded filtration, with no monotonicity condition
on the sequence $\{b_n\}$. Each $I_n$ is integrally closed, and
\begin{equation}\label{eq:surface-components}
\cvec(I_n;R)=(0,n+1,(n+2)b_n).
\end{equation}
Therefore $c_0(\mathcal I;R)=0$ and
$c_1(\mathcal I;R)=1$, whereas the normalized top component
$c_2(I_n;R)/n^2$ is governed by the asymptotic behavior of $b_n/n$.
\end{example}

\begin{corollary}\label{ex:realization}
For every $0\leq A\leq B\leq+\infty$, there is a graded filtration of nonzero integrally closed proper ideals in $k[[x,y]]$ such that $\sqrt{I_n}=(x)$, $\htop I_n=1$, $\ell(I_n)=2$, $c_0(\I;R)=0$, $c_1(\I;R)=1$, $\lowc_2(\I;R)=A$, and $\upc_2(\I;R)=B$.

\end{corollary}

\begin{proof}
By Example~\ref{ex:surface-family}, it is enough to choose positive
integers $b_n$ with
$\liminf_{n\to\infty}b_n/n=A$ and
$\limsup_{n\to\infty}b_n/n=B$, since
\[
\frac{c_2(I_n;R)}{n^2}
=
\left(1+\frac{2}{n}\right)\frac{b_n}{n}.
\]

If $A<\infty$, take
$b_n=\max\{1,\lceil An\rceil\}$ for even $n$. For odd $n$, take
$b_n=\max\{1,\lceil Bn\rceil\}$ when $B<\infty$, and $b_n=n^2$
when $B=\infty$. Then the even and odd subsequences have limits
$A$ and $B$, respectively, and hence
$\lowc_2(\mathcal I;R)=A$ and
$\upc_2(\mathcal I;R)=B$. If $A=B=\infty$, take $b_n=n^2$ for every $n$. All the remaining
properties follow from Example~\ref{ex:surface-family}.
\end{proof}

\begin{remark}
\label{ex:non-noetherian-algebra}
The family algebra in Example~\ref{ex:surface-family} is never
Noetherian. Indeed, if it were generated in degrees at most $N$, then
every homogeneous element of degree $n$ would have $x$-order at least
$(1+1/N)n$, whereas $x^{n+1}y^{b_n}\in I_n$ has $x$-order $n+1$,
a contradiction for $n>N$.

Moreover, every real valuation $v$ of $\operatorname{Frac}(R)$
nonnegative on $R$ satisfies
$v(I_n)=(n+1)v(x)+\min\{v(x),b_nv(y)\}$ and hence
$\lim_{n\to\infty}v(I_n)/n=v(x)$. Thus the asymptotic valuation orders
are independent of the sequence $(b_n)$, although the higher
multiplicity components are not. Finally,
$(x)^{n+2}\subseteq I_n\subseteq(x)^{n+1}$ shows that such
containments do not yield a squeeze principle for $c_2$; in
particular, $c_2(I_n;R)$ need not be $O(n^2)$.
\end{remark}

\subsection*{Common-model convergence in the analytic case}

We next give a geometric convergence criterion in the complex analytic setting. By Achilles--Rams, the positive components of the multiplicity sequence agree with the Segre numbers of Gaffney--Gassler, which admit intersection-theoretic formulas on suitable blowups. On a common model, the divisors associated with the family lie in a fixed finite-dimensional space; their coefficients are subadditive, and Fekete's lemma together with multilinearity of intersection products yields convergence. This viewpoint is related to the recent work of Cutkosky and Monta\~no \cite{CutMontanoFamilies} on graded $\mathfrak m$-primary families, while the result below treats all positive components for arbitrary ideals.
\begin{theorem}
\label{thm:common-analytic-model-convergence}
Let $(X,0)$ be a reduced complex analytic germ of pure dimension
$d\geq1$, let $R=\mathcal O_{X,0}$, and let
$\mathcal I=\{I_n\}_{n\geq0}$ be a graded family of proper ideals of
$R$. Assume that $\operatorname{ht}I_1>0$ and that there exists a
proper modification $\pi\colon Y\longrightarrow X$
such that $I_n\mathcal O_Y$ is invertible for every $n\geq1$. Then the
limit
\[
c_i(\mathcal I;R)
:=
\lim_{n\to\infty}\frac{c_i(I_n;R)}{n^i}
\]
exists and is finite for every $1\leq i\leq d$.
\end{theorem}
\begin{proof}
After replacing $Y$ by the reduced union of the components dominating
those of $X$, blowing up $\mathfrak m\mathcal O_Y$, and normalizing,
we may assume that $Y$ is normal and that both
$\mathfrak m\mathcal O_Y$ and every $I_n\mathcal O_Y$ are invertible.
These operations preserve properness of the modification and
invertibility of the $I_n\mathcal O_Y$. Thus there are effective
Cartier divisors $F,D_1,D_2,\ldots$ such that
\begin{equation}
\label{eq:common-model-divisors}
\mathfrak m\mathcal O_Y=\mathcal O_Y(-F),
\qquad
I_n\mathcal O_Y=\mathcal O_Y(-D_n).
\end{equation}
The divisor $F$ is supported on $\pi^{-1}(0)$. The inclusions $I_mI_n\subseteq I_{m+n}$ give
$D_{m+n}\leq D_m+D_n$, while $I_1^n\subseteq I_n$ gives
$0\leq D_n\leq nD_1$. Let $E_1,\ldots,E_r$ be the prime divisors
appearing in $D_1$. Then
\begin{equation}
\label{eq:Dn-fixed-support}
[D_n]=\sum_{\nu=1}^r a_{\nu,n}[E_\nu],
\qquad
0\leq a_{\nu,n}\leq na_{\nu,1},
\end{equation}
and $a_{\nu,m+n}\leq a_{\nu,m}+a_{\nu,n}$. Hence Fekete's lemma gives
finite limits
\begin{equation}
\label{eq:asymptotic-divisor-coefficients}
\alpha_\nu
=
\lim_{n\to\infty}\frac{a_{\nu,n}}{n}
\qquad(1\leq\nu\leq r).
\end{equation}

Let $V$ be the real vector space generated by the Cartier divisors
$D_n$. Since $Y$ is normal, a Cartier divisor is determined by its
associated Weil divisor, and \eqref{eq:Dn-fixed-support} embeds $V$
into the finite-dimensional space spanned by
$E_1,\ldots,E_r$. Hence $V$ is finite-dimensional. By
\eqref{eq:asymptotic-divisor-coefficients}, the sequence $D_n/n$
converges in $V$ to an element $D_\infty$ whose Weil coefficients are
$\alpha_1,\ldots,\alpha_r$.

For a divisor supported on $E_1\cup\cdots\cup E_r$, write
$[D]^{\mathrm{vert}}$ for the sum of the components mapping to $0$
and $[D]^{\mathrm{hor}}$ for the remaining components. Put
$H=c_1(\mathcal O_Y(-F))$ and
$L(D)=c_1(\mathcal O_Y(-D))$.
Since $\mathfrak m\mathcal O_Y$ and $I_n\mathcal O_Y$ are invertible,
the universal property of blowups gives morphisms from $Y$ to
$\operatorname{Bl}_{\mathfrak m}X$,
$\operatorname{Bl}_{I_n}X$, and
$\operatorname{Bl}_{\mathfrak m I_n}X$. Thus the intersection
formulas of Gaffney--Gassler
\cite[(3.2.3) and (3.2.4)]{GaffneyGassler}, pulled back to $Y$ and
combined with the projection formula, give
\begin{equation}
\label{eq:common-model-lower-components}
c_i(I_n;R)
=
\deg_{\pi^{-1}(0)}
\left(
H^{d-i-1}L(D_n)^{i-1}
\cap\bigl(F\cdot[D_n]^{\mathrm{hor}}\bigr)
\right)
\end{equation}
for $1\leq i<d$, and
\begin{equation}
\label{eq:common-model-top-component}
c_d(I_n;R)
=
\deg_{\pi^{-1}(0)}
\left(
L(D_n)^{d-1}\cap[D_n]^{\mathrm{vert}}
\right).
\end{equation}
Here we use the identification of Segre numbers with the
codimension-indexed multiplicity-sequence components
\cite[Theorem~2]{AchillesRams}. The relevant cycles are supported on
the proper fiber $\pi^{-1}(0)$, so their degrees depend only on their
rational-equivalence classes.

For $1\leq i<d$, define
\[
P_i(D)
=
\deg_{\pi^{-1}(0)}
\left(
H^{d-i-1}L(D)^{i-1}
\cap\bigl(F\cdot[D]^{\mathrm{hor}}\bigr)
\right),
\]
and define
\[
P_d(D)
=
\deg_{\pi^{-1}(0)}
\left(
L(D)^{d-1}\cap[D]^{\mathrm{vert}}
\right).
\]
Since $D\mapsto L(D)$ and the vertical and horizontal projections are
linear, multilinearity of intersection products shows that $P_i$ is a
homogeneous polynomial function of degree $i$ on $V$. By
\eqref{eq:common-model-lower-components} and
\eqref{eq:common-model-top-component},
$c_i(I_n;R)=P_i(D_n)$ for every $n\geq1$. Therefore
\[
\frac{c_i(I_n;R)}{n^i}
=
P_i\left(\frac{D_n}{n}\right).
\]
Since $D_n/n\to D_\infty$ in the finite-dimensional space $V$,
continuity of $P_i$ gives
\[
\lim_{n\to\infty}\frac{c_i(I_n;R)}{n^i}
=
P_i(D_\infty),
\]
which is finite. This proves the assertion for every
$1\leq i\leq d$.
\end{proof}

\begin{remark}
\label{rem:common-model-scope}
The preceding argument applies to reduced, pure-dimensional complex
analytic local rings. Extending it to arbitrary equidimensional,
universally catenary Noetherian local rings would require algebraic
counterparts of the fiber-supported intersection formulas
\eqref{eq:common-model-lower-components}--
\eqref{eq:common-model-top-component}, as well as a treatment of
nilpotent contributions. We therefore keep the analytic result
separate from the corresponding general algebraic question.
\end{remark}

We now turn to the Noetherian case. Finite generation allows us to pass to a standard Veronese subalgebra and control each residue class of degrees by a fixed factor, which forces convergence of every asymptotic component.
\begin{theorem}
\label{thm:noetherian-convergence}
Let \(\mathcal I=\{I_n\}_{n\geq0}\)
be a graded filtration of proper ideals, and assume that \(\mathscr R(\mathcal I) = \bigoplus_{n\geq0}I_nt^n\)
is finitely generated over $R$. Then, for every $0\leq i\leq d$,
\(
\lowc_i(\mathcal I;R)
=
\upc_i(\mathcal I;R)
<
+\infty.
\)
More precisely, there exists an integer $q\geq1$ such that
\begin{equation}
\label{eq:limit-formula}
c_i(\mathcal I;R)
=
\lim_{n\to\infty}
\frac{c_i(I_n;R)}{n^i}
=
\frac{c_i(I_q;R)}{q^i}
\qquad(0\leq i\leq d).
\end{equation}
\end{theorem}

\begin{proof}
Since $\mathscr R(\mathcal I)$ is a finitely generated graded
$R$-algebra, there exists $q\geq1$ such that the Veronese algebra
$B=\mathscr R(\mathcal I)^{(q)}$ is standard graded
\cite[Lemma~10.56.2]{Stacks}. Setting $J=I_q$, this means that
\begin{equation}
\label{eq:noeth-standard-veronese}
I_{qn}=J^n
\qquad(n\geq0).
\end{equation}

For $0\leq r<q$, let
$M^{(r)}=\bigoplus_{n\geq0}I_{qn+r}t^{qn+r}$, viewed after regrading
as a graded $B$-module. The algebra $\mathscr R(\mathcal I)$ is
integral over $B$: if $z$ is homogeneous, then $z^q\in B$. Since
$\mathscr R(\mathcal I)$ is of finite type over $B$ and $B$ is
Noetherian, it is finite over $B$. Hence each $M^{(r)}$ is a finite
graded $B$-module. As $B$ is standard, there exists $a_r$ such that
$M^{(r)}_{n+1}=B_1M^{(r)}_n$ for all $n\geq a_r$. Thus, with
$K_r=I_{qa_r+r}$,
\begin{equation}
\label{eq:residue-factor}
I_{qn+r}=J^{\,n-a_r}K_r
\qquad(n\geq a_r).
\end{equation}
For $r=0$, we may take $a_0=0$, so
\eqref{eq:noeth-standard-veronese} applies directly. Since $\mathcal I$ is a filtration,
$I_1^m\subseteq I_m\subseteq I_1$ for every $m\geq1$. Hence
\begin{equation}
\label{eq:noeth-fixed-radical}
\sqrt{I_m}=\sqrt{I_1}
\qquad(m\geq1).
\end{equation}
In particular, $J$ and each $K_r$ with $r>0$ have the same radical.
For $i=0$, Proposition~\ref{prop:classical}\textup{(ii)} and
\eqref{eq:noeth-fixed-radical} give
$c_0(I_m;R)=c_0(I_1;R)$ for every $m\geq1$. Hence
$c_0(\mathcal I;R)=c_0(I_q;R)$.

Now fix $1\leq i\leq d$. For $r=0$, homogeneity and
\eqref{eq:noeth-standard-veronese} give
$c_i(I_{qn};R)=n^ic_i(J;R)$. For $1\leq r<q$,
Lemma~\ref{lem:fixed-factor} and \eqref{eq:residue-factor} give
\[
c_i(I_{qn+r};R)
=
(n-a_r)^ic_i(J;R)+O(n^{i-1}).
\]
Since $(n-a_r)^i=n^i+O(n^{i-1})$ and
$(qn+r)^i=q^in^i+O(n^{i-1})$, it follows that
\[
\lim_{n\to\infty}
\frac{c_i(I_{qn+r};R)}{(qn+r)^i}
=
\frac{c_i(J;R)}{q^i}
\]
for every residue class $r$ modulo $q$. Hence the full sequence
converges to this same value. Since $J=I_q$,
\[
c_i(\mathcal I;R)
=
\frac{c_i(I_q;R)}{q^i}
\qquad(1\leq i\leq d).
\]
Together with the case $i=0$, this proves
\eqref{eq:limit-formula} and the finiteness of all asymptotic
components.
\end{proof}


Fix an integer $q\geq1$ such that
$I_{qv}=I_q^v$ for every $v\geq0$, whose existence follows from
finite generation of $\mathscr R(\mathcal I)$
\cite[Lemma~10.56.2]{Stacks}. Set
\[
G_q(\mathcal I)
:=
\bigoplus_{v\geq0}I_{qv}/I_{q(v+1)}
=
\operatorname{gr}_{I_q}(A),
\qquad
\mathfrak N_q
:=
\mathfrak m/I_q\oplus
\bigoplus_{v>0}I_q^v/I_q^{v+1}.
\]
Thus $\mathfrak N_q$ is the homogeneous maximal ideal of
$G_q(\mathcal I)$.

The following result extends
\cite[Proposition~2.5 and Corollary~2.6]{AM97}
from a single ideal to Noetherian filtrations.

\begin{proposition}
\label{prop:total-multiplicity-family}
Let $(A,\mathfrak m)$ be a $d$-dimensional Noetherian local ring and
let $\mathcal I$ be a Noetherian graded filtration. Then
\begin{equation}
\label{eq:total-multiplicity-family}
e_{\mathfrak N_q}\bigl(G_q(\mathcal I)\bigr)
=
\sum_{i=0}^{d}q^ic_i(\mathcal I;A).
\end{equation}
In particular, if $\mathcal I$ is standard, equivalently
$I_n=I_1^n$ for every $n\geq0$, then
\[
e_{\mathfrak N_1}\bigl(G_1(\mathcal I)\bigr)
=
\sum_{i=0}^{d}c_i(\mathcal I;A).
\]
\end{proposition}

\begin{proof}
Since $G_q(\mathcal I)=\operatorname{gr}_{I_q}(A)$, the classical
total multiplicity formula
\cite[Proposition~2.5 and Corollary~2.6]{AM97} gives
\[
e_{\mathfrak N_q}\bigl(G_q(\mathcal I)\bigr)
=
\sum_{i=0}^{d}c_i(I_q;A).
\]
By Theorem~\ref{thm:noetherian-convergence},
$c_i(I_q;A)=q^ic_i(\mathcal I;A)$ for every $0\leq i\leq d$, which
gives \eqref{eq:total-multiplicity-family}. The standard case follows
by taking $q=1$.
\end{proof}

\begin{example}
\label{ex:total-multiplicity-non-noetherian}
Let $R=k[[x]]$, $\mathfrak m=(x)$, and set
$I_0=R$ and $I_n=\mathfrak m$ for every $n\geq1$. Then
$\mathcal I$ is a graded filtration. For every $q\geq1$,
\[
G_q(\mathcal I)
=
\bigoplus_{v\geq0}I_{qv}/I_{q(v+1)}
\cong R/\mathfrak m=k,
\]
concentrated in degree zero. Hence
$e_{(0)}(G_q(\mathcal I))=1$. On the other hand, since $\mathfrak m$ is $\mathfrak m$-primary in
the one-dimensional regular local ring $R$, one has
$c_0(I_n;R)=0$ and $c_1(I_n;R)=1$ for every $n\geq1$. Therefore
$c_0(\mathcal I;R)=c_1(\mathcal I;R)=0$, and hence
\[
e_{(0)}(G_q(\mathcal I))
=
1
\neq
0
=
\sum_{i=0}^1q^ic_i(\mathcal I;R).
\]

Moreover, no $q$ is a standard Veronese index, since
$I_{2q}=\mathfrak m\neq\mathfrak m^2=I_q^2$. Thus
$\mathscr R(\mathcal I)$ is not finitely generated, by
\cite[Lemma~10.56.2]{Stacks}. Thus the total multiplicity formula of Proposition~\ref{prop:total-multiplicity-family} may fail completely without finite generation.
\end{example}

\section{Integral closure and Rees criteria}
\label{sec:closure}

For graded families, three closure operations naturally arise:
termwise integral closure, integral closure of the family algebra in
$R[t]$, and closure defined by asymptotic valuation inequalities.
These constructions need not coincide.
For a graded family $\mathcal I$, define
\[
\mathcal I^{\mathrm{term}}
=
\{\overline{I_n}\}_{n\geq0},
\qquad
\overline{\mathscr R(\mathcal I)}^{\,R[t]}
=
\bigoplus_{n\geq0}I_n^{\mathrm{alg}}t^n.
\]
Here the first bar denotes integral closure of an ideal, while the
second denotes integral closure of the algebra relative to $R[t]$.
The latter is graded by \cite[Theorem~2.3.2]{HunekeSwanson}. Thus
$f\in I_n^{\mathrm{alg}}$ means that $ft^n$ is integral over
$\mathscr R(\mathcal I)$.

\begin{proposition}
\label{prop:comparison-closures}
The family $\mathcal I^{\mathrm{term}}$ is graded, and it is a
filtration whenever $\mathcal I$ is a filtration. The family
$\mathcal I^{\mathrm{alg}}$ is a graded filtration, with
$I_n^{\mathrm{alg}}$ proper for every $n>0$. Moreover, $\mathcal I
\subseteq
\mathcal I^{\mathrm{term}}
\subseteq
\mathcal I^{\mathrm{alg}}$
degreewise. If $\mathcal I=\{I^n\}_{n\geq0}$ is adic, then
$\mathcal I^{\mathrm{term}}
=\mathcal I^{\mathrm{alg}}
=\{\overline{I^n}\}_{n\geq0}$.
\end{proposition}

\begin{proof}
The standard product formula
$\overline A\,\overline B\subseteq\overline{AB}$, together with
$I_mI_n\subseteq I_{m+n}$, gives
$\overline{I_m}\,\overline{I_n}\subseteq\overline{I_{m+n}}$.
If $\mathcal I$ is a filtration, then
$I_{n+1}\subseteq I_n$, and monotonicity of integral closure gives
$\overline{I_{n+1}}\subseteq\overline{I_n}$. Thus
$\mathcal I^{\mathrm{term}}$ is graded and, when appropriate,
decreasing.

Clearly $\mathcal I\subseteq\mathcal I^{\mathrm{term}}$. If
$f\in\overline{I_n}$, choose an equation
$f^s+a_1f^{s-1}+\cdots+a_s=0$ with $a_j\in I_n^j$. Since
$I_n^j\subseteq I_{jn}$, multiplication by $t^{ns}$ gives a monic
integral equation for $ft^n$ over $\mathscr R(\mathcal I)$. Hence
$\overline{I_n}\subseteq I_n^{\mathrm{alg}}$.
By \cite[Theorem~2.3.2]{HunekeSwanson}, the relative integral closure
of $\mathscr R(\mathcal I)$ in $R[t]$ is graded, so
$\mathcal I^{\mathrm{alg}}$ is a graded family. To see that it is
decreasing, suppose that $ft^{n+1}$ is integral over
$\mathscr R(\mathcal I)$. Then
$(ft^n)^{n+1}=f(ft^{n+1})^n$ is integral over
$\mathscr R(\mathcal I)$. Since an element whose positive power is
integral is itself integral, $ft^n$ is integral over
$\mathscr R(\mathcal I)$. Thus
$I_{n+1}^{\mathrm{alg}}\subseteq I_n^{\mathrm{alg}}$. Now let $n>0$ and $f\in I_n^{\mathrm{alg}}$. A homogeneous integral
equation for $ft^n$ has the form
\[
(ft^n)^s+a_1t^n(ft^n)^{s-1}+\cdots+a_st^{ns}=0,
\qquad
a_j\in I_{jn}.
\]
After cancelling $t^{ns}$, we obtain
$f^s+a_1f^{s-1}+\cdots+a_s=0$. Since every $I_{jn}$ is proper and
hence contained in $\mathfrak m$, reduction modulo $\mathfrak m$
gives $\overline f^{\,s}=0$ in $R/\mathfrak m$. Therefore
$f\in\mathfrak m$, so $I_n^{\mathrm{alg}}$ is proper.

If $\mathcal I=\{I^n\}$, then a homogeneous element $ft^n$
is integral over $R[It]$ if and only if $f$ is integral over $I^n$:
a homogeneous integral equation for $ft^n$ has coefficients in
$I^{jn}t^{jn}$, and cancelling the powers of $t$ gives an
integral equation over $I^n$, and conversely. Hence
$I_n^{\mathrm{alg}}=\overline{I^n}$ for every $n$.
\end{proof}

\begin{remark}
\label{rem:termwise-invariance}
Under the usual hypotheses for integral-closure invariance of the
multiplicity sequence, for instance when $R$ is equidimensional and
universally catenary, the equality
$\overline{I_n}=\overline{J_n}$ for all $n\gg0$ implies equality of
all lower and upper asymptotic components, by the corresponding
one-ideal result \cite{AM06,PTUV}. This is a termwise statement and
does not imply
$\mathcal I^{\mathrm{term}}=\mathcal I^{\mathrm{alg}}$, nor
invariance under algebraic closure for arbitrary non-Noetherian
families.
\end{remark}

\begin{proposition}
\label{prop:veronese-closure}
Let $\mathcal I$ be a Noetherian graded family and let $q$ be a
standard Veronese index. Then, for every $n\geq1$,
\begin{equation}
\label{eq:algebraic-root-criterion}
f\in I_n^{\mathrm{alg}}
\quad\Longleftrightarrow\quad
f^q\in\overline{I_q^{\,n}}.
\end{equation}
In particular,
$I_{qv}^{\mathrm{alg}}=\overline{I_{qv}}$ for every $v\geq1$.
\end{proposition}

\begin{proof}
Set $A=\mathscr R(\mathcal I)$ and
$C=A^{(q)}=R[I_qt^q]$. Since $A$ is integral over its Veronese
subalgebra $C$, the two algebras have the same integral closure
relative to $R[t]$. If $f\in I_n^{\mathrm{alg}}$, then $ft^n$ is integral over $C$, and
hence so is
$(ft^n)^q=f^qt^{qn}$. Viewing $C=R[I_qt^q]$ as the Rees algebra of
$I_q$ with variable $t^q$, Proposition~\ref{prop:comparison-closures}
gives $f^q\in\overline{I_q^n}$.

Conversely, if $f^q\in\overline{I_q^n}$, then
$(ft^n)^q$ is integral over $C$. Hence $ft^n$ is integral over
$C[(ft^n)^q]$, which is integral over $C$; by transitivity,
$ft^n$ is integral over $C$. Thus $f\in I_n^{\mathrm{alg}}$.
Taking $n=qv$ in
\eqref{eq:algebraic-root-criterion} gives
$f\in I_{qv}^{\mathrm{alg}}$ if and only if
$f^q\in\overline{I_q^{qv}}$. By the power criterion for integral
closure, this is equivalent to
$f\in\overline{I_q^v}$. Since $q$ is a standard Veronese index,
$I_{qv}=I_q^v$, and therefore
$I_{qv}^{\mathrm{alg}}=\overline{I_{qv}}$.
\end{proof}

Note that the finite generation allows us to choose a common standard Veronese index
for finitely many family algebras. This reduces integral dependence
between Noetherian filtrations to the corresponding one-ideal
criterion of Polini--Trung--Ulrich--Validashti.

\begin{theorem}
\label{thm:family-Rees-criterion}
Let $R$ be equidimensional and universally catenary, and let
$\I\subseteq\J$ be Noetherian graded filtrations of proper ideals.
The following are equivalent:
\begin{enumerate}[label=\textup{(\roman*)}]
\item $\mathscr R(\J)$ is integral over $\mathscr R(\I)$;
\item $c_i(\I;R)\leq c_i(\J;R)$ for every $0\leq i\leq d$;
\item $\cvec(\I;R)=\cvec(\J;R)$;
\item $\overline{\mathscr R(\I)}^{\,R[t]}
       =\overline{\mathscr R(\J)}^{\,R[t]}$.
\end{enumerate}
\end{theorem}

\begin{proof}
Choose $q\geq1$ which is a standard Veronese index for both
$\mathscr R(\mathcal I)$ and $\mathscr R(\mathcal J)$, and set
$A=\mathscr R(\mathcal I)$ and $B=\mathscr R(\mathcal J)$. Then $A^{(q)}=R[I_qt^q]$, and  $B^{(q)}=R[J_qt^q].$
Both $A$ and $B$ are integral over their $q$th Veronese subalgebras.
If $B$ is integral over $A$, then $B$ is finite over $A$ and hence
over $A^{(q)}$; since $A^{(q)}$ is Noetherian,
$B^{(q)}$ is finite, and therefore integral, over $A^{(q)}$.
Conversely, if $B^{(q)}$ is integral over $A^{(q)}$, then the
integrality of $B$ over $B^{(q)}$ shows that $B$ is integral over
$A^{(q)}$, and hence over $A$. Thus $B$  is integral over $A
$ if and only if
$B^{(q)}$ is integral over $A^{(q)}$.
Since these are the Rees algebras of $I_q$ and $J_q$ with variable
$t^q$, this is equivalent to
$J_q\subseteq\overline{I_q}$.

By \cite[Theorem~4.2]{PTUV}, the latter condition is equivalent both
to
$c_i(I_q;R)\leq c_i(J_q;R)$ for every $0\leq i\leq d$
and to
$c_i(I_q;R)=c_i(J_q;R)$ for every $0\leq i\leq d$.
Theorem~\ref{thm:noetherian-convergence} gives
\[
c_i(\mathcal I;R)=\frac{c_i(I_q;R)}{q^i},
\qquad
c_i(\mathcal J;R)=\frac{c_i(J_q;R)}{q^i},
\]
so \textup{(i)}, \textup{(ii)}, and \textup{(iii)} are equivalent. Hence, for $A\subseteq B\subseteq R[t]$, the algebra $B$ is
integral over $A$ if and only if their integral closures relative to
$R[t]$ coincide. Hence \textup{(i)} is equivalent to \textup{(iv)}.
\end{proof}

\subsection*{Divisorial saturation}

Throughout this subsection, let $R$ be an excellent local domain.
Write $\operatorname{Div}(R)$ for the normalized divisorial valuations
of $\operatorname{Frac}(R)$ that are nonnegative on $R$. Excellence
ensures that the normalized blowups are Noetherian and that the Rees
valuations of ideals belong to this class; see
\cite[Chapter~10]{HunekeSwanson}. The center of
$v\in\operatorname{Div}(R)$ is
$\mathfrak c_R(v)=\{f\in R:v(f)>0\}$, which need not be $\mathfrak m$. For a graded family $\mathcal I$ of nonzero proper ideals, set
\[
v(\mathcal I)
=
\inf_{n\geq1}\frac{v(I_n)}n
=
\lim_{n\to\infty}\frac{v(I_n)}n,
\qquad
v(I_n)=\min\{v(f):0\neq f\in I_n\}.
\]
The limit exists by Fekete's lemma, since
$v(I_{m+n})\leq v(I_m)+v(I_n)$. Define
$\widetilde I^{\,\mathrm{div}}_0=R$ and, for $n\geq1$,
\begin{equation}
\label{eq:div-saturation}
\widetilde I^{\,\mathrm{div}}_n
=
\{f\in\mathfrak m:
v(f)\geq n\,v(\mathcal I)
\text{ for every }v\in\operatorname{Div}(R)\}.
\end{equation}
The condition $f\in\mathfrak m$ ensures that the positive-degree
terms remain proper even when all asymptotic valuation orders vanish.

\begin{proposition}
\label{prop:noeth-div-alg}
The family $\widetilde{\mathcal I}^{\,\mathrm{div}}$ is a graded
filtration containing $\mathcal I^{\mathrm{alg}}$, and its family
algebra is integrally closed relative to $R[t]$. If $\mathcal I$ is
Noetherian, then
\[
\mathcal I
\subseteq
\mathcal I^{\mathrm{term}}
\subseteq
\mathcal I^{\mathrm{alg}}
=
\widetilde{\mathcal I}^{\,\mathrm{div}}.
\]
\end{proposition}

\begin{proof}
The defining valuation inequalities immediately give
$\widetilde I_m^{\,\mathrm{div}}
\widetilde I_n^{\,\mathrm{div}}
\subseteq
\widetilde I_{m+n}^{\,\mathrm{div}}$
and
$\widetilde I_{n+1}^{\,\mathrm{div}}
\subseteq
\widetilde I_n^{\,\mathrm{div}}$.
Moreover, if $f\in I_n$, then
$v(f)\geq v(I_n)\geq n\,v(\mathcal I)$ for every
$v\in\operatorname{Div}(R)$; since $I_n$ is proper,
$f\in\mathfrak m$. Hence
$\mathcal I\subseteq\widetilde{\mathcal I}^{\,\mathrm{div}}$.
Set
$S=\mathscr R(\widetilde{\mathcal I}^{\,\mathrm{div}})$.
To prove that $S$ is integrally closed relative to $R[t]$, it is
enough to consider a homogeneous element $ft^n\in R[t]$ integral
over $S$. For $n>0$, a homogeneous monic equation gives
\[
f^r+a_1f^{r-1}+\cdots+a_r=0,
\qquad
a_j\in\widetilde I_{jn}^{\,\mathrm{div}}.
\]
Since $a_j\in\mathfrak m$, reduction modulo $\mathfrak m$ gives
$f\in\mathfrak m$. Suppose $f\neq0$ and that, for some
$v\in\operatorname{Div}(R)$,
$v(f)<n\,v(\mathcal I)$. Then, for every $1\leq j\leq r$,
\[
v(a_jf^{r-j})
\geq
jn\,v(\mathcal I)+(r-j)v(f)
>
rv(f).
\]
Thus $f^r$ is the unique term of minimal $v$-value in the integral
relation, which is impossible. Hence
$v(f)\geq n\,v(\mathcal I)$ for every divisorial valuation $v$, so
$f\in\widetilde I_n^{\,\mathrm{div}}$. Therefore $S$ is integrally
closed relative to $R[t]$. Since $S$ contains
$\mathscr R(\mathcal I)$, it also contains its relative integral
closure, and consequently
$\mathcal I^{\mathrm{alg}}
\subseteq\widetilde{\mathcal I}^{\,\mathrm{div}}$.

Now, assume that $\mathcal I$ is Noetherian and choose a standard
Veronese index $q$. Then $I_{qv}=I_q^v$ for every $v\geq0$, and hence,
for every $v\in\operatorname{Div}(R)$,
\[
v(\mathcal I)
=
\lim_{m\to\infty}\frac{v(I_m)}m
=
\frac{v(I_q)}q.
\]
If $f\in\widetilde I_n^{\,\mathrm{div}}$, then
$qv(f)\geq n\,v(I_q)$ for every $v\in\operatorname{Div}(R)$, and in
particular for every Rees valuation of $I_q$. By the valuative
criterion for integral closure,
$f^q\in\overline{I_q^n}$. Proposition~\ref{prop:veronese-closure}
then gives $f\in I_n^{\mathrm{alg}}$. Thus
$\widetilde I_n^{\,\mathrm{div}}\subseteq I_n^{\mathrm{alg}}$ for
every $n$, and the reverse inclusion was proved above. The remaining
inclusions follow from Proposition~\ref{prop:comparison-closures}.
\end{proof}

\begin{corollary}
\label{thm:noeth-valuative-rees}
Let $\mathcal I\subseteq\mathcal J$ be Noetherian graded filtrations
of nonzero proper ideals in an excellent local domain $R$. Then
$\cvec(\mathcal I;R)=\cvec(\mathcal J;R)$ if and only if
$\widetilde{\mathcal I}^{\,\mathrm{div}}
=
\widetilde{\mathcal J}^{\,\mathrm{div}}$.
\end{corollary}

\begin{proof}
Since $R$ is a domain, it is equidimensional, and excellence implies
that it is universally catenary. Hence
Theorem~\ref{thm:family-Rees-criterion} gives
$\cvec(\mathcal I;R)=\cvec(\mathcal J;R)$ if and only if $\overline{\mathscr R(\mathcal I)}^{\,R[t]}
=
\overline{\mathscr R(\mathcal J)}^{\,R[t]}.$
By definition, the homogeneous components of these relative integral
closures are $\mathcal I^{\mathrm{alg}}$ and
$\mathcal J^{\mathrm{alg}}$, respectively. Since both filtrations are
Noetherian, Proposition~\ref{prop:noeth-div-alg} gives
$\mathcal I^{\mathrm{alg}}
=\widetilde{\mathcal I}^{\,\mathrm{div}}$ and
$\mathcal J^{\mathrm{alg}}
=\widetilde{\mathcal J}^{\,\mathrm{div}}$.
The assertion follows.
\end{proof}

\begin{remark}
The equality
$\mathcal I^{\mathrm{alg}}=\widetilde{\mathcal I}^{\,\mathrm{div}}$
may fail without the Noetherian hypothesis. Indeed, let
$R=k[[x]]$ and $I_n=(x^{n+1})$ for $n\geq1$. Then
$v(\mathcal I)=v(x)$, so
$\widetilde I_n^{\,\mathrm{div}}=(x^n)$. However,
$x\notin I_1^{\mathrm{alg}}$. Otherwise $xt$ would satisfy a
homogeneous monic equation over $\mathscr R(\mathcal I)$, which,
after cancelling the appropriate power of $t$, would have the form $x^r+a_1x^{r-1}+\cdots+a_r=0$, where
 $a_j\in(x^{j+1})$.
The first term has $x$-order $r$, while every other term has
$x$-order at least $r+1$, a contradiction. Thus
$I_1^{\mathrm{alg}}\subsetneq
\widetilde I_1^{\,\mathrm{div}}$.
\end{remark}

\subsection*{Comparison with the primary case}

For an arbitrary Noetherian local ring, let $\mathcal V(R)$ denote
Cutkosky's set of $\mathfrak m$-valuations \cite{Cut}. For a graded
$\mathfrak m$-primary family $\mathcal I$, set
\[
\widetilde I_n^{\,\mathrm{Cut}}
=
\{f\in\mathfrak m:
v(f)\geq n\,v(\mathcal I)
\text{ for every }v\in\mathcal V(R)\}.
\]
In the excellent local-domain case, this agrees with the divisorial
saturation of the preceding subsection: divisorial valuations whose
center is not $\mathfrak m$ have $v(I_1)=0$ and impose no condition.

If $\mathcal I\subseteq\mathcal J$ are graded
$\mathfrak m$-primary families, then their asymptotic multiplicity
sequences are
$(0,\ldots,0,e(\mathcal I;R))$ and
$(0,\ldots,0,e(\mathcal J;R))$. Hence Cutkosky's asymptotic Rees
theorem \cite[Theorems~7.3--7.5]{Cut} gives
\[
\cvec(\mathcal I;R)=\cvec(\mathcal J;R)
\quad\Longleftrightarrow\quad
v(\mathcal I)=v(\mathcal J)
\ \text{for all }v\in\mathcal V(R)
\quad\Longleftrightarrow\quad
\widetilde{\mathcal I}^{\,\mathrm{Cut}}
=
\widetilde{\mathcal J}^{\,\mathrm{Cut}}.
\]
Thus, in the primary case, the valuative criterion is already
contained in Cutkosky's theory and requires no finite generation of
the family algebras.

\section{Diagonal mixed multiplicity sequences}
\label{sec:mixed-noetherian-filtrations}

Throughout this section, $(R,\mathfrak m)$ is a Noetherian local ring
of dimension $d\geq1$. We begin with the mixed multiplicity sequence
of Callejas--Bedregal and Jorge P\'erez \cite{CBJP}.

\subsection*{The fixed-ideal construction}

Let $\mathbf J=(J_1,\ldots,J_s)$, write
$J^{\mathbf u}=J_1^{u_1}\cdots J_s^{u_s}$ and
$\mathbf1=(1,\ldots,1)$, and set
$B=R[J_1T_1,\ldots,J_sT_s]$. The diagonal multiform algebra is
\[
\mathscr S^\Delta(\mathbf J)
=
B/(J_1\cdots J_s)B
=
\bigoplus_{\mathbf u\in\mathbb N^s}
\frac{J^{\mathbf u}}{J^{\mathbf u+\mathbf1}}.
\]
Its $\mathfrak m$-associated graded algebra has components
\[
\bigl[\operatorname{gr}_{\mathfrak m}
\mathscr S^\Delta(\mathbf J)\bigr]_{r,\mathbf u}
=
\frac{\mathfrak m^rJ^{\mathbf u}+J^{\mathbf u+\mathbf1}}
{\mathfrak m^{r+1}J^{\mathbf u}+J^{\mathbf u+\mathbf1}}.
\]
Summing in the $\mathfrak m$-direction gives
\begin{equation}
\label{eq:diagonal-mixed-Hilbert-function}
H^\Delta_{\mathbf J}(r,\mathbf u)
=
\lambda_R\!\left(
\frac{J^{\mathbf u}}
{\mathfrak m^{r+1}J^{\mathbf u}+J^{\mathbf u+\mathbf1}}
\right).
\end{equation}
By the multigraded Hilbert-polynomial construction of
\cite[Sections~2 and~5]{CBJP}, this function agrees for large
$(r,\mathbf u)$ with a polynomial of total degree at most $d-1$.
We write its homogeneous part of degree $d-1$ as
\begin{equation}
\label{eq:diagonal-mixed-Hilbert-polynomial}
\sum_{i=1}^d\ \sum_{|\boldsymbol\alpha|=i-1}
\frac{
c^\Delta_{i,\boldsymbol\alpha}
(J_1|\cdots|J_s;R)}
{(d-i)!\boldsymbol\alpha!}
r^{d-i}\mathbf u^{\boldsymbol\alpha},
\end{equation}
where
$|\boldsymbol\alpha|=\sum_j\alpha_j$ and
$\boldsymbol\alpha!=\prod_j\alpha_j!$.

\begin{definition}
\label{def:diagonal-mixed-sequence}
The coefficients in
\eqref{eq:diagonal-mixed-Hilbert-polynomial} are the
\emph{diagonal mixed multiplicity-sequence components}. In the
notation of \cite[Definition~5.1]{CBJP}, 
$c^\Delta_{i,\boldsymbol\alpha}
=
c^{\mathrm{CBJP}}_{d-i,\boldsymbol\alpha}.$
\end{definition}

Thus only the indexing and the superscript $\Delta$ are new. When
$s=1$, one has
$\mathscr S^\Delta(J_1)=\operatorname{gr}_{J_1}(R)$ and
$c^\Delta_{i,(i-1)}(J_1;R)=c_i(J_1;R)$. The main identity we use is
\cite[Theorem~5.3]{CBJP}, written in codimension indexing:
\begin{equation}
\label{eq:fixed-mixed-product}
c_i(J_1\cdots J_s;R)
=
\sum_{|\boldsymbol\alpha|=i-1}
\binom{i-1}{\boldsymbol\alpha}
c^\Delta_{i,\boldsymbol\alpha}
(J_1|\cdots|J_s;R),
\qquad 1\leq i\leq d.
\end{equation}
Applied to $J_1^{a_1},\ldots,J_s^{a_s}$, it gives
\begin{equation}
\label{eq:powered-mixed-product}
c_i(J_1^{a_1}\cdots J_s^{a_s};R)
=
\sum_{|\boldsymbol\alpha|=i-1}
\binom{i-1}{\boldsymbol\alpha}
c^\Delta_{i,\boldsymbol\alpha}
(J_1^{a_1}|\cdots|J_s^{a_s};R).
\end{equation}
At this stage, the powers on the right should be retained; no separate
monomial scaling law for the diagonal coefficients is being assumed. Indeed, if $K=J_1\cdots J_s$, then
$H_K(r,v)=H^\Delta_{\mathbf J}(r,v\mathbf1)$. Comparing the
coefficients of $r^{d-i}v^{i-1}$ in the corresponding top
homogeneous parts gives \eqref{eq:fixed-mixed-product}; the same
argument applied to the powered tuple gives
\eqref{eq:powered-mixed-product}.

\begin{lemma}
\label{lem:mixed-homogeneity}
For $a\geq1$ and $|\boldsymbol\alpha|=i-1$,
\[
c^\Delta_{i,\boldsymbol\alpha}
(J_1^a|\cdots|J_s^a;R)
=
a^i
c^\Delta_{i,\boldsymbol\alpha}
(J_1|\cdots|J_s;R).
\]
\end{lemma}

\begin{proof}
For each multidegree $\mathbf u$, filter
$J^{a\mathbf u}/J^{a\mathbf u+a\mathbf1}$ by $J^{a\mathbf u}
\supseteq
J^{a\mathbf u+\mathbf1}
\supseteq\cdots\supseteq
J^{a\mathbf u+a\mathbf1}.$
The successive quotients are $\frac{J^{a\mathbf u+b\mathbf1}}
     {J^{a\mathbf u+(b+1)\mathbf1}}$, for $0\leq b<a.$
For each fixed $b$, these quotients form the $a$th multiveronese
restriction of a fixed shift of
$\mathscr S^\Delta(\mathbf J)$. Let $P(r,\mathbf u)$ denote the homogeneous part of degree $d-1$ in
\eqref{eq:diagonal-mixed-Hilbert-polynomial}. The same Artin--Rees
additivity argument as in Lemma~\ref{lem:additivity}, applied in the
multigraded setting, shows that the top homogeneous part associated
with the above filtration is $\sum_{b=0}^{a-1}P(r,a\mathbf u+b\mathbf1).$
Since translation by the fixed vector $b\mathbf1$ changes only terms
of total degree strictly smaller than $d-1$, its degree-$(d-1)$ part is $
aP(r,a\mathbf u).$
Hence the coefficient of
$r^{d-i}\mathbf u^{\boldsymbol\alpha}$ is multiplied by
$a^{1+|\boldsymbol\alpha|}=a^i$, because
$|\boldsymbol\alpha|=i-1$. This proves the assertion.
\end{proof}

\subsection*{Families and the primary comparison}

Let $\mathcal I^{(1)},\ldots,\mathcal I^{(s)}$ be graded families whose
positive-degree terms have positive height. For
$|\boldsymbol\alpha|=i-1$, define
\[
\lowc^\Delta_{i,\boldsymbol\alpha}
(\mathcal I^{(1)}|\cdots|\mathcal I^{(s)};R)
=
\liminf_{n\to\infty}
\frac{
c^\Delta_{i,\boldsymbol\alpha}
(I_n^{(1)}|\cdots|I_n^{(s)};R)}
{n^i},
\]
and
\[
\upc^\Delta_{i,\boldsymbol\alpha}
(\mathcal I^{(1)}|\cdots|\mathcal I^{(s)};R)
=
\limsup_{n\to\infty}
\frac{
c^\Delta_{i,\boldsymbol\alpha}
(I_n^{(1)}|\cdots|I_n^{(s)};R)}
{n^i}.
\]
When these two values agree, their common value is denoted by
$c^\Delta_{i,\boldsymbol\alpha}
(\mathcal I^{(1)}|\cdots|\mathcal I^{(s)};R)$.
No convergence is assumed in this definition.

\begin{theorem}
\label{thm:mixed-primary}
Let $\mathcal I^{(1)},\ldots,\mathcal I^{(s)}$ be graded
$\mathfrak m$-primary families. Then all diagonal asymptotic mixed
components exist and are finite. They vanish for $i<d$, while for
$|\boldsymbol\alpha|=d-1$,
\begin{equation}
\label{eq:mixed-primary-comparison}
c^\Delta_{d,\boldsymbol\alpha}
(\mathcal I^{(1)}|\cdots|\mathcal I^{(s)};R)
=
\sum_{j=1}^s
e\bigl(
(\mathcal I^{(1)})^{[\alpha_1]},\ldots,
(\mathcal I^{(j)})^{[\alpha_j+1]},\ldots,
(\mathcal I^{(s)})^{[\alpha_s]};R
\bigr),
\end{equation}
where the terms on the right are Cutkosky's mixed multiplicities.
\end{theorem}

\begin{proof}
We first consider fixed $\mathfrak m$-primary ideals
$J_1,\ldots,J_s$. Set $K=J_1\cdots J_s$. Since
$KJ^{\mathbf u}=J^{\mathbf u+\mathbf1}$ and $K$ is
$\mathfrak m$-primary, there exists $c\geq1$ such that
$\mathfrak m^c\subseteq K$. Hence, for every $\mathbf u$ and
$r\geq c-1$,
\[
H^\Delta_{\mathbf J}(r,\mathbf u)
=
\lambda_R\left(
\frac{J^{\mathbf u}}{J^{\mathbf u+\mathbf1}}
\right)
=
\lambda_R(R/J^{\mathbf u+\mathbf1})
-
\lambda_R(R/J^{\mathbf u}).
\]
Thus the eventual polynomial is independent of $r$, and therefore
$c^\Delta_{i,\boldsymbol\alpha}(J_1|\cdots|J_s;R)=0$ for every
$i<d$. Let $Q(\mathbf u)$ be the Bhattacharya polynomial of
$J_1,\ldots,J_s$. Its homogeneous part of degree $d$ is
\[
Q_d(\mathbf u)
=
\sum_{|\boldsymbol\beta|=d}
\frac{
e(J_1^{[\beta_1]},\ldots,J_s^{[\beta_s]};R)}
{\boldsymbol\beta!}
\mathbf u^{\boldsymbol\beta}.
\]
The degree-$(d-1)$ part of
$Q(\mathbf u+\mathbf1)-Q(\mathbf u)$ is
$\sum_{j=1}^s\partial Q_d/\partial u_j$. Comparing the coefficient
of $\mathbf u^{\boldsymbol\alpha}/\boldsymbol\alpha!$, for
$|\boldsymbol\alpha|=d-1$, gives
\begin{equation}
\label{eq:fixed-primary-diagonal}
c^\Delta_{d,\boldsymbol\alpha}(J_1|\cdots|J_s;R)
=
\sum_{j=1}^s
e(J_1^{[\alpha_1]},\ldots,
J_j^{[\alpha_j+1]},\ldots,
J_s^{[\alpha_s]};R).
\end{equation}
This is also \cite[Proposition~5.2]{CBJP} in the present indexing. Applying \eqref{eq:fixed-primary-diagonal} with
$J_j=I_n^{(j)}$, and using
\cite[Theorem~1.6 and equation~(9)]{Cut}, we obtain, for every
$|\boldsymbol\beta|=d$,
\[
\lim_{n\to\infty}
\frac{
e((I_n^{(1)})^{[\beta_1]},\ldots,
(I_n^{(s)})^{[\beta_s]};R)}
{n^d}
=
e((\mathcal I^{(1)})^{[\beta_1]},\ldots,
(\mathcal I^{(s)})^{[\beta_s]};R),
\]
and this limit is finite. Dividing
\eqref{eq:fixed-primary-diagonal} by $n^d$ and passing to the limit
in the finite sum gives \eqref{eq:mixed-primary-comparison}.
\end{proof}

\begin{remark}
For $s>1$, a diagonal coefficient is generally a sum of adjacent
Rees--Teissier mixed multiplicities rather than a single one. Indeed,
in the $\mathfrak m$-primary case,
\[
c^\Delta_{d,\boldsymbol\alpha}
(J_1^{a_1}|\cdots|J_s^{a_s};R)
=
\sum_{j=1}^s
a_j\prod_{\ell=1}^s a_\ell^{\alpha_\ell}\,
e(J_1^{[\alpha_1]},\ldots,
J_j^{[\alpha_j+1]},\ldots,
J_s^{[\alpha_s]};R).
\]
Thus simultaneous rescaling $a_1=\cdots=a_s=a$ gives the factor
$a^d$, while separate rescaling does not in general produce a single
monomial factor.
\end{remark}

\subsection*{The fixed-factor problem}
Simultaneous homogeneity controls the standard Veronese subsequence,
but not the remaining residue classes of a Noetherian filtration.
The missing ingredient is the following fixed-factor estimate. For
tuples $\mathbf J=(J_1,\ldots,J_s)$ and
$\mathbf K=(K_1,\ldots,K_s)$, consider
\begin{equation}
\label{eq:mixed-FF}
c^\Delta_{i,\boldsymbol\alpha}
(J_1^mK_1|\cdots|J_s^mK_s;R)
=
m^i
c^\Delta_{i,\boldsymbol\alpha}
(J_1|\cdots|J_s;R)
+
O(m^{i-1}).
\end{equation}
Below we only need the case
$\sqrt{J_j}=\sqrt{K_j}$, with these ideals of positive height. For $s=1$, \eqref{eq:mixed-FF} is precisely
Lemma~\ref{lem:fixed-factor}. It also holds when all the ideals are
$\mathfrak m$-primary. Indeed, the components with $i<d$ vanish, while
for $i=d$ formula \eqref{eq:fixed-primary-diagonal} and the
multilinearity of ordinary mixed Hilbert--Samuel multiplicities show
that the terms involving only the $J_j$ contribute the degree-$d$
part in $m$, whereas every term involving at least one fixed factor
$K_j$ has degree at most $d-1$.

\begin{proposition}
\label{thm:mixed-noetherian-convergence}
Let $\mathcal I^{(1)},\ldots,\mathcal I^{(s)}$ be Noetherian graded
filtrations with $\operatorname{ht}I_1^{(j)}>0$ for every $j$.
Choose a common standard Veronese index $q$, and assume that
\eqref{eq:mixed-FF} holds for the fixed-factor tuples arising from
the corresponding residue-class decompositions. Then, for every
$1\leq i\leq d$ and $|\boldsymbol\alpha|=i-1$, the diagonal mixed
limit exists and is finite, and
\[
c^\Delta_{i,\boldsymbol\alpha}
(\mathcal I^{(1)}|\cdots|\mathcal I^{(s)};R)
=
\frac{
c^\Delta_{i,\boldsymbol\alpha}
(I_q^{(1)}|\cdots|I_q^{(s)};R)}
{q^i}.
\]
\end{proposition}

\begin{proof}
Set $J_j=I_q^{(j)}$. For each $0\leq r<q$, the residue-class module
of $\mathscr R(\mathcal I^{(j)})$ is finite over its standard
Veronese algebra. Since there are only finitely many pairs $(j,r)$,
there exists $a\geq1$ such that, for every $j$, every $r$, and every
$m\geq a$, $I_{qm+r}^{(j)}
=
J_j^{\,m-a}K_{j,r}$, and $K_{j,r}:=I_{qa+r}^{(j)}.$
Because each $\mathcal I^{(j)}$ is a filtration,
$\sqrt{K_{j,r}}=\sqrt{J_j}=\sqrt{I_1^{(j)}}$.

Fix $1\leq i\leq d$ and $|\boldsymbol\alpha|=i-1$. For
$1\leq r<q$, applying \eqref{eq:mixed-FF} with $m-a$ in place of
$m$ gives
\[
c^\Delta_{i,\boldsymbol\alpha}
(I_{qm+r}^{(1)}|\cdots|I_{qm+r}^{(s)};R)
=
(m-a)^i
c^\Delta_{i,\boldsymbol\alpha}
(J_1|\cdots|J_s;R)
+
O(m^{i-1}).
\]
Since $(m-a)^i=m^i+O(m^{i-1})$ and
$(qm+r)^i=q^im^i+O(m^{i-1})$, it follows that
\[
\lim_{m\to\infty}
\frac{
c^\Delta_{i,\boldsymbol\alpha}
(I_{qm+r}^{(1)}|\cdots|I_{qm+r}^{(s)};R)}
{(qm+r)^i}
=
\frac{
c^\Delta_{i,\boldsymbol\alpha}
(J_1|\cdots|J_s;R)}
{q^i}.
\]
For $r=0$, one has $I_{qm}^{(j)}=J_j^m$, and
Lemma~\ref{lem:mixed-homogeneity} gives the same limit directly.
Thus every residue-class subsequence modulo $q$ converges to the same
finite value. Since these subsequences exhaust the full sequence, the
diagonal mixed limit exists and
\[
c^\Delta_{i,\boldsymbol\alpha}
(\mathcal I^{(1)}|\cdots|\mathcal I^{(s)};R)
=
q^{-i}
c^\Delta_{i,\boldsymbol\alpha}
(I_q^{(1)}|\cdots|I_q^{(s)};R).
\]
\end{proof}

The intersection formula of
Lemma~\ref{lem:mixed-intersection-formula} proves
\eqref{eq:mixed-FF} in the reduced pure-dimensional complex analytic
setting. Indeed, principalize the finitely many ideals
$J_j,K_j$ and $\mathfrak m$ on a common model. If $D_{J_j}$ and
$D_{K_j}$ are the corresponding divisors, then
\[
D_{J_j^mK_j}=mD_{J_j}+D_{K_j}.
\]
Multilinearity of the degree-$i$ intersection formula shows that the
degree-$i$ term in $m$ is
$m^ic^\Delta_{i,\boldsymbol\alpha}
(J_1|\cdots|J_s;R)$, while every term involving at least one
$D_{K_j}$ has degree at most $i-1$. Thus
\eqref{eq:mixed-FF} holds in this setting.

\begin{problem}
\label{prob:mixed-fixed-factor}
Establish \eqref{eq:mixed-FF} for arbitrary positive-height tuples
with $\sqrt{J_j}=\sqrt{K_j}$, or determine additional hypotheses
under which it holds.
\end{problem}

This problem is not resolved by separate rescaling in
\eqref{eq:fixed-mixed-product}; accordingly,
Proposition~\ref{thm:mixed-noetherian-convergence} is conditional
outside the cases established above.

\subsection*{Products of families}
For graded families
$\mathcal I^{(1)},\ldots,\mathcal I^{(s)}$, define their
\emph{termwise product}
$\mathcal K=\mathcal I^{(1)}\cdots\mathcal I^{(s)}$ by
$K_n=I_n^{(1)}\cdots I_n^{(s)}$. Then $\mathcal K$ is graded, and it
is a filtration whenever all the factors are filtrations.

\begin{proposition}
\label{thm:asymptotic-product}
Fix $1\leq i\leq d$. If all diagonal mixed components of index $i$
exist and are finite, then $c_i(\mathcal K;R)$ exists and
\begin{equation}
\label{eq:asymptotic-product}
c_i(\mathcal K;R)
=
\sum_{|\boldsymbol\alpha|=i-1}
\binom{i-1}{\boldsymbol\alpha}
c^\Delta_{i,\boldsymbol\alpha}
(\mathcal I^{(1)}|\cdots|\mathcal I^{(s)};R).
\end{equation}
\end{proposition}

\begin{proof}
For every $n\geq1$, applying
\eqref{eq:fixed-mixed-product} to
$I_n^{(1)},\ldots,I_n^{(s)}$ gives
\[
c_i(K_n;R)
=
\sum_{|\boldsymbol\alpha|=i-1}
\binom{i-1}{\boldsymbol\alpha}
c^\Delta_{i,\boldsymbol\alpha}
(I_n^{(1)}|\cdots|I_n^{(s)};R).
\]
Dividing by $n^i$ and passing to the limit in the finite sum yields
\eqref{eq:asymptotic-product}.
\end{proof}

\begin{remark}
If the factors are Noetherian, then so is their termwise product:
$\mathscr R(\mathcal K)$ is the diagonal subalgebra of the finitely
generated multigraded algebra
$\bigoplus_{\mathbf n}\prod_j I_{n_j}^{(j)}\mathbf T^{\mathbf n}$,
and hence is finitely generated. Thus
Theorem~\ref{thm:noetherian-convergence} gives convergence of the
ordinary components of $\mathcal K$ independently of
\eqref{eq:mixed-FF}, although it does not separate the summands in
\eqref{eq:asymptotic-product}.

For $\mathfrak m$-primary families,
Theorem~\ref{thm:mixed-primary} makes
\eqref{eq:asymptotic-product} unconditional and yields Cutkosky's
product formula
\[
e(\mathcal I^{(1)}\cdots\mathcal I^{(s)};R)
=
\sum_{|\boldsymbol\beta|=d}
\binom{d}{\boldsymbol\beta}
e((\mathcal I^{(1)})^{[\beta_1]},\ldots,
(\mathcal I^{(s)})^{[\beta_s]};R).
\]
Indeed,
$\sum_{j:\beta_j>0}(d-1)!/(\boldsymbol\beta-\mathbf e_j)!
=d!/\boldsymbol\beta!$.
\end{remark}

\subsection*{A geometric direction}

For a fixed ideal, \cite[Theorem~5.12]{CRPU} interprets
$c_i(I;R)$ as the local multiplicity of a Segre-type cycle obtained
from general hyperplanes. In using this description, one must
distinguish cycles from their classes modulo rational equivalence.

\begin{remark}
\label{rem:local-degree-caution}
The map
$\sum_{\mathfrak p}a_{\mathfrak p}[R/\mathfrak p]
\mapsto
\sum_{\mathfrak p}a_{\mathfrak p}e(R/\mathfrak p)$
does not descend to the ordinary Chow group of a local scheme.
Indeed, on $\operatorname{Spec}k[[x]]$ the closed point is the
principal divisor of $x$, hence rationally equivalent to zero, while
its local multiplicity is $1$. Thus Chow-theoretic projection formulas
alone do not yield numerical formulas for the $c_i$; the relevant
local intersection data must also be retained. See also
\cite{CutMontano}.
\end{remark}

The common-model argument requires an intersection formula for the
diagonal coefficients. The next lemma derives it directly from the
multigraded quotient, using additivity of leading Hilbert terms,
multidegrees on the proper fiber, and the projection formula. Its
specialization to one ideal agrees with the formulas of
Gaffney--Gassler and Achilles--Rams.

Throughout, let $R=\mathcal O_{X,0}$, where $(X,0)$ is a reduced
complex analytic germ of pure dimension $d\geq1$. For proper ideals
$\mathbf J=(J_1,\ldots,J_s)$, we use the diagonal coefficients
$c_{i,\boldsymbol\alpha}^{\Delta}(J_1|\cdots|J_s;R)$ defined by
\[
H_{\mathbf J}^{\Delta}(r,\mathbf u)
=
\lambda_R\left(
\frac{J^{\mathbf u}}
{\mathfrak m^{r+1}J^{\mathbf u}+J^{\mathbf u+\mathbf1}}
\right),
\]
whose degree-$(d-1)$ homogeneous part is
\[
\sum_{i=1}^d\ \sum_{|\boldsymbol\alpha|=i-1}
\frac{
c_{i,\boldsymbol\alpha}^{\Delta}(J_1|\cdots|J_s;R)}
{(d-i)!\,\boldsymbol\alpha!}
r^{d-i}\mathbf u^{\boldsymbol\alpha}.
\]

\begin{lemma}
\label{lem:mixed-intersection-formula}
Let $(X,0)$ be a reduced complex analytic germ of pure dimension
$d\geq1$, and set $R=\mathcal O_{X,0}$. Let $J_1,\ldots,J_s$ be
proper ideals of positive height. Suppose that a proper modification
$\pi:Y\to X$ principalizes $\mathfrak m,J_1,\ldots,J_s$, and write
\[
\mathfrak m\mathcal O_Y=\mathcal O_Y(-F),
\qquad
J_a\mathcal O_Y=\mathcal O_Y(-D_a).
\]
Set $D=\sum_aD_a$, $H=c_1(\mathcal O_Y(-F))$, and
$L_a=c_1(\mathcal O_Y(-D_a))$. Denote by $[D]^0$ the part of
$[D]$ supported on $\pi^{-1}(0)$, and put
$[D]^\circ=[D]-[D]^0$. Then, for $1\leq i\leq d$ and
$|\boldsymbol\alpha|=i-1$,
\begin{equation}
\label{eq:mixed-intersection-formula}
c^\Delta_{i,\boldsymbol\alpha}(J_1|\cdots|J_s;R)
=
\begin{cases}
\displaystyle
\deg_{\pi^{-1}(0)}
\left(
H^{d-i-1}\prod_{a=1}^sL_a^{\alpha_a}
\cap(F\cdot[D]^\circ)
\right),& i<d,\\[2mm]
\displaystyle
\deg_{\pi^{-1}(0)}
\left(
\prod_{a=1}^sL_a^{\alpha_a}\cap[D]^0
\right),& i=d.
\end{cases}
\end{equation}
\end{lemma}

\begin{proof}
Set $K=J_1\cdots J_s$, $B=\bigoplus_{\mathbf u\in\mathbb N^s}
J^{\mathbf u}\mathbf T^{\mathbf u}$, and $M=B/KB$.
Then
\[
H_M(r,\mathbf u)
=
\lambda_R\left(
\frac{J^{\mathbf u}}
{\mathfrak m^{r+1}J^{\mathbf u}+KJ^{\mathbf u}}
\right),
\]
and the diagonal coefficients are precisely the coefficients of the
homogeneous part of total degree $d-1$ of its eventual polynomial.

Let $Z=\operatorname{MultiProj}(B)$. After discarding components not
dominating components of $X$, we may work on the reduced
pure-dimensional part of $Z$. Each $J_a\mathcal O_Z$ is invertible;
write
\[
J_a\mathcal O_Z=\mathcal O_Z(-A_a),
\qquad
A=\sum_aA_a,
\qquad
\ell_a=c_1(\mathcal O_Z(-A_a)).
\]
Since $K\mathcal O_Z=\mathcal O_Z(-A)$, the sheaf associated with
$M$ is $\mathcal O_A$. The degree-$(d-1)$ part of the cumulative Hilbert polynomial is
additive on finite multigraded exact sequences. Indeed, the same
Artin--Rees argument as in Lemma~\ref{lem:additivity} shows that the
induced and intrinsic $\mathfrak m$-adic filtrations differ only by
a bounded shift, whose contribution has total degree at most $d-2$.
A homogeneous prime filtration of $M$ therefore gives
\begin{equation}
\label{eq:mixed-prime-decomposition}
H_M^{[d-1]}
=
\sum_W\mu_W H_{C_W}^{[d-1]},
\end{equation}
where
$[A]=\sum_W\mu_W[W]$,
$C_W=B/P_W$, and
$\mu_W=\lambda_{\mathcal O_{Z,\eta_W}}
(\mathcal O_{A,\eta_W})$.
Components of dimension less than $d-1$, as well as irrelevant
components, do not contribute to \eqref{eq:mixed-prime-decomposition}.

Suppose first that $W\subseteq Z_0$, where $Z_0$ is the fiber over
$0$. Then $\mathfrak mC_W=0$, so $H_{C_W}(r,\mathbf u)$ is
independent of $r$. Hence $W$ contributes only when $i=d$, and the
multidegree interpretation of the leading Hilbert coefficients gives
\[
\mu_W
\deg_W\left(
\prod_a\ell_a^{\alpha_a}\cap[W]
\right),
\qquad |\boldsymbol\alpha|=d-1.
\]

Now suppose that $W\not\subseteq Z_0$. Let
$q_W:\widehat W\to W$ be the blowup of
$\mathfrak m\mathcal O_W$, with exceptional divisor $E_W$, and set
$h_W=c_1(\mathcal O_{\widehat W}(-E_W))$. Since
\[
\dim_{\mathbb C}
\bigl[\operatorname{gr}_{\mathfrak m}(C_W)\bigr]_{r,\mathbf u}
=
H_{C_W}(r,\mathbf u)-H_{C_W}(r-1,\mathbf u),
\]
a term
\[
\frac{a_{k,\boldsymbol\alpha}}
{k!\boldsymbol\alpha!}
r^k\mathbf u^{\boldsymbol\alpha},
\qquad
k+|\boldsymbol\alpha|=d-1,
\]
in the top homogeneous part of $H_{C_W}$ contributes
\[
\frac{a_{k,\boldsymbol\alpha}}
{(k-1)!\boldsymbol\alpha!}
r^{k-1}\mathbf u^{\boldsymbol\alpha}
\]
to the leading polynomial of
$\operatorname{gr}_{\mathfrak m}(C_W)$ when $k\geq1$.
The associated multiprojective scheme is the exceptional divisor
$E_W$, with tautological classes $h_W$ and $q_W^*\ell_a$.
Therefore
\begin{equation}
\label{eq:mixed-component-degree}
a_{k,\boldsymbol\alpha}
=
\deg_{E_W}\left(
h_W^{k-1}
\prod_a(q_W^*\ell_a)^{\alpha_a}
\cap[E_W]
\right).
\end{equation}
There is no degree-$(d-1)$ contribution with $k=0$, since
$W\cap Z_0$ is a proper closed subset of the $(d-1)$-dimensional
irreducible space $W$, and hence has dimension at most $d-2$. Let
\[
b:Z'=\operatorname{Bl}_{\mathfrak m\mathcal O_Z}(Z)\longrightarrow Z
\]
and denote its exceptional divisor by $F'$. Put
$H'=c_1(\mathcal O_{Z'}(-F'))$. If $W\not\subseteq Z_0$, its strict
transform in $Z'$ is
$\operatorname{Bl}_{\mathfrak m\mathcal O_W}(W)$, so
\eqref{eq:mixed-component-degree}, together with
\eqref{eq:mixed-prime-decomposition}, gives, for $i<d$,
\[
c^\Delta_{i,\boldsymbol\alpha}(\mathbf J;R)
=
\deg_{Z'_0}\left(
(H')^{d-i-1}
\prod_a(b^*\ell_a)^{\alpha_a}
\cap\bigl(F'\cdot[b^*A]^\circ\bigr)
\right).
\]
The components contained in $Z_0$ similarly give
\[
c^\Delta_{d,\boldsymbol\alpha}(\mathbf J;R)
=
\deg_{Z_0}\left(
\prod_a\ell_a^{\alpha_a}\cap[A]^0
\right).
\]

Since $Y$ principalizes
$\mathfrak m,J_1,\ldots,J_s$, the universal properties of the
multiblowup and of the blowup of $\mathfrak m$ give morphisms $Y\longrightarrow Z'$,
$Y\longrightarrow Z.$ The tautological classes pull back to $H,L_1,\ldots,L_s$, while
$D$ is the pullback of $A$. At the level of divisor cycles,
the horizontal and vertical parts push forward to
$[b^*A]^\circ$ and $[A]^0$, respectively; exceptional components
whose images have smaller dimension have zero pushforward.
The projection formula therefore transforms the two preceding
expressions into
\eqref{eq:mixed-intersection-formula}. All resulting zero-cycles are supported on the proper fiber
$\pi^{-1}(0)$, so their degrees are well defined.
\end{proof}

\begin{remark}
For $s=1$, Lemma~\ref{lem:mixed-intersection-formula} recovers
\cite[(3.2.3)--(3.2.4)]{GaffneyGassler}, together with the
identification of Segre numbers and codimension-indexed
multiplicity-sequence components from
\cite[Theorem~2]{AchillesRams}. For $s>1$, the argument works
directly with the diagonal multigraded quotient and does not require
directional coefficients.
\end{remark}

The preceding intersection formula gives a mixed analogue of the
common-model convergence theorem. Once all ideals are principalized
on a single modification, their divisors have fixed finite support
and subadditive coefficients. Fekete's lemma then produces limiting
divisors, while multilinearity of the intersection formula gives
convergence of every diagonal mixed component.

\begin{theorem}
\label{thm:mixed-common-model-convergence}
Let $(X,0)$ be a reduced complex analytic germ of pure dimension
$d\geq1$, and set $R=\mathcal O_{X,0}$. For $1\leq a\leq s$, let
$\mathcal I^{(a)}=\{I_n^{(a)}\}_{n\geq0}$ be a graded family of
proper ideals with $\operatorname{ht}I_1^{(a)}>0$. Suppose that there
exists a proper modification $\pi:Y\to X$ such that
$I_n^{(a)}\mathcal O_Y$ is invertible for every $a$ and every
$n\geq1$. Then, for every $1\leq i\leq d$ and
$|\boldsymbol\alpha|=i-1$, the limit
\[
c^\Delta_{i,\boldsymbol\alpha}
(\mathcal I^{(1)}|\cdots|\mathcal I^{(s)};R)
=
\lim_{n\to\infty}
\frac{
c^\Delta_{i,\boldsymbol\alpha}
(I_n^{(1)}|\cdots|I_n^{(s)};R)}
{n^i}
\]
exists and is finite.
\end{theorem}

\begin{proof}
After replacing $Y$ by the reduced union of the components dominating
those of $X$, blowing up $\mathfrak m\mathcal O_Y$, and normalizing,
we may assume that $Y$ is normal and that
$\mathfrak m\mathcal O_Y$ and every
$I_n^{(a)}\mathcal O_Y$ are invertible. Thus $\mathfrak m\mathcal O_Y=\mathcal O_Y(-F)$, $I_n^{(a)}\mathcal O_Y=\mathcal O_Y(-D_{a,n}),$
for effective Cartier divisors $F$ and $D_{a,n}$.
For each $a$, the graded-family condition gives
$D_{a,m+n}\leq D_{a,m}+D_{a,n}$, while
$I_1^{(a)n}\subseteq I_n^{(a)}$ gives
$0\leq D_{a,n}\leq nD_{a,1}$. Hence all the $D_{a,n}$ are supported
on the finitely many prime divisors
$E_1,\ldots,E_r$ occurring in $\sum_aD_{a,1}$. Write $D_{a,n}
=
\sum_{\nu=1}^r b_{a,\nu,n}E_\nu.$
Then each sequence $\{b_{a,\nu,n}\}_{n\geq1}$ is nonnegative and
subadditive. By Fekete's lemma,
\begin{equation}
\label{eq:mixed-divisor-limits}
\gamma_{a,\nu}
=
\lim_{n\to\infty}\frac{b_{a,\nu,n}}n
=
\inf_{n\geq1}\frac{b_{a,\nu,n}}n
\end{equation}
exists and is finite. Let $V$ be the real vector space generated by all the Cartier divisors
$D_{a,n}$. Since $Y$ is normal, the coefficient map embeds $V$ into
the finite-dimensional space spanned by $E_1,\ldots,E_r$. Thus, for
each $a$, $\frac{D_{a,n}}n\longrightarrow\delta_a$
 in $V,$
where $\delta_a$ has coefficients
$\gamma_{a,1},\ldots,\gamma_{a,r}$.

Fix $1\leq i\leq d$ and
$|\boldsymbol\alpha|=i-1$. Put
$H=c_1(\mathcal O_Y(-F))$, and extend
$D\mapsto c_1(\mathcal O_Y(-D))$ real-linearly to $V$; denote this
map by $\ell$. The decompositions of a divisor cycle into the parts
supported over $0$ and away from $0$ also extend linearly to $V$.

For $\mathbf v=(v_1,\ldots,v_s)\in V^s$, set
$v=v_1+\cdots+v_s$ and define
\[
P_{i,\boldsymbol\alpha}(\mathbf v)
=
\begin{cases}
\displaystyle
\deg_{\pi^{-1}(0)}
\left(
H^{d-i-1}\prod_a\ell(v_a)^{\alpha_a}
\cap(F\cdot[v]^\circ)
\right),& i<d,\\[2mm]
\displaystyle
\deg_{\pi^{-1}(0)}
\left(
\prod_a\ell(v_a)^{\alpha_a}\cap[v]^0
\right),& i=d.
\end{cases}
\]
By multilinearity of intersection products,
$P_{i,\boldsymbol\alpha}$ is a polynomial function on $V^s$,
homogeneous of degree
$1+|\boldsymbol\alpha|=i$ under simultaneous scaling.
Lemma~\ref{lem:mixed-intersection-formula} gives
\[
c^\Delta_{i,\boldsymbol\alpha}
(I_n^{(1)}|\cdots|I_n^{(s)};R)
=
P_{i,\boldsymbol\alpha}
(D_{1,n},\ldots,D_{s,n}).
\]
Therefore
\[
\frac{
c^\Delta_{i,\boldsymbol\alpha}
(I_n^{(1)}|\cdots|I_n^{(s)};R)}
{n^i}
=
P_{i,\boldsymbol\alpha}
\left(
\frac{D_{1,n}}n,\ldots,\frac{D_{s,n}}n
\right).
\]
Since $D_{a,n}/n\to\delta_a$ in $V$ for every $a$, continuity of
$P_{i,\boldsymbol\alpha}$ yields
\[
\lim_{n\to\infty}
\frac{
c^\Delta_{i,\boldsymbol\alpha}
(I_n^{(1)}|\cdots|I_n^{(s)};R)}
{n^i}
=
P_{i,\boldsymbol\alpha}
(\delta_1,\ldots,\delta_s),
\]
which is finite.
\end{proof}

\begin{remark}
Theorem~\ref{thm:common-analytic-model-convergence} is the case
$s=1$ of Theorem~\ref{thm:mixed-common-model-convergence}. Indeed,
then $|\boldsymbol\alpha|=i-1$ forces
$\boldsymbol\alpha=(i-1)$, and
$c^\Delta_{i,i-1}(J;R)=c_i(J;R)$.
Thus the mixed theorem extends the one-family common-model result.
\end{remark}

\begin{corollary}
\label{cor:mixed-family-product}
Under the hypotheses of
Proposition~\ref{thm:mixed-noetherian-convergence}, set $\mathcal P
=
\{I_n^{(1)}\cdots I_n^{(s)}\}_{n\geq0}.$
Then $\mathcal P$ is a Noetherian graded filtration and, for every
$1\leq i\leq d$,
\begin{equation}
\label{eq:asymptotic-mixed-product}
c_i(\mathcal P;R)
=
\sum_{|\boldsymbol\alpha|=i-1}
\binom{i-1}{\boldsymbol\alpha}
c^\Delta_{i,\boldsymbol\alpha}
(\mathcal I^{(1)}|\cdots|\mathcal I^{(s)};R).
\end{equation}
\end{corollary}

\begin{proof}
The termwise product of Noetherian graded filtrations is again a
Noetherian graded filtration: its family algebra is the diagonal
subalgebra of the finitely generated multigraded algebra $\bigoplus_{\mathbf n\in\mathbb N^s}
I_{n_1}^{(1)}\cdots I_{n_s}^{(s)}\mathbf t^{\mathbf n},$
and is therefore finitely generated. By
Proposition~\ref{thm:mixed-noetherian-convergence}, all diagonal mixed
components exist and are finite. The formula now follows directly
from Proposition~\ref{thm:asymptotic-product}.
\end{proof}

\begin{remark}
\label{rem:primary-mixed-cutkosky}
Theorem~\ref{thm:mixed-primary} requires neither finite generation of
the family algebras nor a common birational model. Its convergence
statement follows directly from Cutkosky's mixed multiplicities for
graded $\mathfrak m$-primary families
\cite[Theorem~1.6 and equation~(9)]{Cut}. In particular, for
$|\boldsymbol\alpha|=d-1$,
\[
c^\Delta_{d,\boldsymbol\alpha}
(\mathcal I^{(1)}|\cdots|\mathcal I^{(s)};R)
=
\sum_{a=1}^{s}
e\bigl(
(\mathcal I^{(1)})^{[\alpha_1]},\ldots,
(\mathcal I^{(a)})^{[\alpha_a+1]},\ldots,
(\mathcal I^{(s)})^{[\alpha_s]};R
\bigr).
\]
\end{remark}

The analytic common-model theorem suggests the following algebraic
extension. The main missing ingredient is a sufficiently robust local
intersection theory on the fixed birational model that recovers the
positive multiplicity-sequence components.

\begin{problem}
\label{prob:common-model}
Let $R$ be an excellent local domain and let $\mathcal I$ be a graded
family of nonzero proper ideals. Suppose that there exists a proper
birational morphism
$Y\to\operatorname{Spec}R$, with $Y$ normal, such that every
$I_n\mathcal O_Y$ is invertible. Under what additional local
intersection-theoretic hypotheses do all positive asymptotic
components $c_i(\mathcal I;R)$ exist and remain finite?
\end{problem}

\section{Acknowledgments}
This work was developed during the postdoctoral fellowship of the first author at the Instituto de Ci\^encias Matem\'aticas e de Computa\c{c}\~ao (ICMC), Universidade de S\~ao Paulo (USP) - 2026. He gratefully acknowledges financial support from FAPESP, grant 2025/20830-5. The first author also gratefully acknowledges the Universidade Tecnol\'ogica Federal do Paran\'a -- Campus Guarapuava for the opportunity to undertake his postdoctoral research. OpenAI’s ChatGPT was used during the preparation of this work as an auxiliary generative-AI tool for  discussion, checks of algebraic manipulations and examples, comparison with cited literature, and assistance with English exposition and LaTeX. All material used in the final manuscript was checked by the authors against the relevant arguments and sources. The authors assume full responsibility for the final text of this paper.


\begin{thebibliography}{99}

\bibitem{AchillesRams}
R.~Achilles and S.~Rams,
\emph{Intersection numbers, Segre numbers and generalized Samuel
multiplicities},
Arch. Math. (Basel) \textbf{77} (2001), 391--398.

\bibitem{GaffneyGassler}
T.~Gaffney and R.~Gassler,
\emph{Segre numbers and hypersurface singularities},
J. Algebraic Geom. \textbf{8} (1999), 695--736.

\bibitem{AM97}
R.~Achilles and M.~Manaresi,
\emph{Multiplicities of a bigraded ring and intersection theory},
Math. Ann. \textbf{309} (1997), no.~4, 573--591.

\bibitem{AM06}
R.~Achilles and M.~Manaresi,
\emph{Generalized Samuel multiplicities and applications},
Rend. Sem. Mat. Univ. Politec. Torino \textbf{64} (2006), no.~4,
345--372.


\bibitem{CBJP}
R.~Callejas--Bedregal and V.~H. Jorge P\'erez,
\emph{Mixed multiplicities for arbitrary ideals and generalized Buchsbaum--Rim multiplicities},
J. London Math. Soc. (2) \textbf{76} (2007), no.~2, 384--398.

\bibitem{CBJPAdd}
R.~Callejas--Bedregal and V.~H. Jorge P\'erez,
\emph{Some properties of the multiplicity sequence for arbitrary ideals},
Rocky Mountain J. Math. \textbf{40} (2010), no.~6, 1809--1827.

\bibitem{CRPU}
Y.~Cid-Ruiz, C.~Polini, and B.~Ulrich,
\emph{Multidegrees, families, and integral dependence},
\href{https://arxiv.org/abs/2405.07000}{arXiv:2405.07000}, 2024.

\bibitem{Cut}
S.~D. Cutkosky,
\emph{Multiplicities of graded families of ideals on Noetherian local rings},
preprint, 2026, \href{https://arxiv.org/abs/2603.06844v2}{arXiv:2603.06844v2}.

\bibitem{CutMontano}
S.~D. Cutkosky and J.~Monta\~no,
\emph{Multiplicities and degree functions in local rings via intersection products},
J. London Math. Soc. (2) \textbf{112} (2025), no.~5, e70338.

\bibitem{CutMontanoFamilies}
S.~D. Cutkosky and J.~Monta\~no,
\emph{Degree functions of graded families of ideals},
Collect. Math. (2026),
doi:10.1007/s13348-026-00519-w.


\bibitem{HunekeSwanson}
C.~Huneke and I.~Swanson,
\emph{Integral Closure of Ideals, Rings, and Modules},
London Mathematical Society Lecture Note Series, vol.~336,
Cambridge University Press, Cambridge, 2006.

\bibitem{PTUV}
C.~Polini, N.~V. Trung, B.~Ulrich, and J.~Validashti,
\emph{Multiplicity sequence and integral dependence},
Math. Ann. \textbf{378} (2020), no.~3--4, 951--969.

\bibitem{Stacks}
The Stacks Project Authors,
\emph{The Stacks Project}, Lemma 10.56.2, Tag 0EGH,
\url{https://stacks.math.columbia.edu/tag/0EGH}.

\end{thebibliography}
\end{document}